\documentclass[12pt,reqno]{amsart}
\usepackage[margin=1in]{geometry}
\usepackage{amsmath,amssymb,amsthm,graphicx,amsxtra, setspace}
\usepackage[utf8]{inputenc}
\usepackage{mathrsfs}
\usepackage{hyperref}
\usepackage{xcolor}
\usepackage{upgreek}
\usepackage{mathtools}
\allowdisplaybreaks
\newtheorem{theorem}{Theorem}[section]
\newtheorem{lem}[theorem]{Lemma}
\newtheorem{rem}[theorem]{Remark}
\newtheorem{cor}[theorem]{Corollary}
\newtheorem{Def}[theorem]{Definition}

\newtheorem{Ex}[theorem]{Example}
\newtheorem{Assumption}[theorem]{Assumption}
\usepackage{latexsym}
\usepackage{hyperref}
\usepackage{graphicx}
\usepackage{epstopdf}

\allowdisplaybreaks

\let\originalleft\left
\let\originalright\right
\renewcommand{\left}{\mathopen{}\mathclose\bgroup\originalleft}
\renewcommand{\right}{\aftergroup\egroup\originalright}

\newcommand{\Addresses}{{
		\footnote{

			\noindent	 \textsuperscript{1,3}Department of Applied Sciences and Engineering, Indian Institute of Technology Roorkee-IIT Roorkee, Roorkee, 247667, India.

			\noindent  \textit{e-mail\textsuperscript{1}:} \texttt{sumit@as.iitr.ac.in.}
			
			\noindent  \textit{e-mail\textsuperscript{3}:} \texttt{jay.dabas@gmail.com.}

			\noindent \textsuperscript{2}Department of Mathematics, Indian Institute of Technology Roorkee-IIT Roorkee, Roorkee, 247667, India.\par\nopagebreak
			\noindent  \textit{e-mail\textsuperscript{2}:} \texttt{manilfma@iitr.ac.in, maniltmohan@gmail.com.}

			\noindent \textsuperscript{*}Corresponding author.

			\textit{Key words:} approximate controllability, impulsive functional differential inclusion, evolution family, mild solution, fixed point theorem.
			
			Mathematics Subject Classification (2010): 34K06, 34A12, 37L05, 93B05.

}}}

\begin{document}
	\title[Non-autonomous evolution inclusions in Banach spaces]{Existence and approximate controllability of non-autonomous functional impulsive evolution inclusions in Banach spaces\Addresses}
	\author [S. Arora, M. T. Mohan and J. Dabas]{S. Arora\textsuperscript{1}, Manil T. Mohan\textsuperscript{2}  and J. dabas\textsuperscript{3*}}
	\maketitle{}
	\begin{abstract}
		In this paper, we are concerned with the approximate controllability results for a class of impulsive functional differential control systems involving time dependent operators in Banach spaces. First, we show the existence of a mild solution for non-autonomous functional impulsive evolution inclusions  in  separable reflexive Banach spaces with the help of the evolution family and a generalization of the Leray-Schauder fixed point theorem for multi-valued maps. In order to establish sufficient conditions for the approximate controllability of our problem, we first consider a linear-quadratic regulator problem and obtain the optimal control in the feedback form, which contains  the resolvent operator consisting of duality mapping.  With the help of this optimal control, we prove the  approximate controllability of the linear system and hence derive sufficient conditions for the approximate controllability of our problem.  Moreover, in this paper, we rectify several shortcomings of the related works available in the literature, namely, proper identification of resolvent operator in Banach spaces, characterization of phase space in the presence of impulsive effects and lack of compactness of the operator $h(\cdot)\mapsto \int_{0}^{\cdot}\mathrm{U}(\cdot,s)h(s)\mathrm{d} s : \mathrm{L}^{1}([0,T];\mathbb{Y}) \rightarrow \mathrm{C}([0,T];\mathbb{Y}),$ where $\mathbb{Y}$ is a Banach space and $\mathrm{U}(\cdot,\cdot)$ is the evolution family, etc. Finally, we provide a concrete example to illustrate the efficiency of our results.
	\end{abstract}
	\section{Introduction}\label{intro}\setcounter{equation}{0}
	There are many evolutionary processes subject to abrupt changes in state occur at certain time instant, that abrupt changes can be well approximated as being in the form of impulses. These processes are characterized by impulsive differential equations, see for example \cite{GA,AM,TA}, etc. On the other hand, there are various chemical, physical or biological processes that are mathematically described by partial functional differential equations with infinite delay, for example, heat conduction in materials with fading memory, reaction-diffusion logistic models with infinite delay, and neural networks etc (cf. \cite{XL,LA,JW} and references therein). In fact, the theory of impulsive functional differential equations is seeking great attention by many researchers and has emerged as an important area of investigation in recent years. Recently, several authors studied the existence results for the impulsive functional differential equations involving finite and infinite delays, see for instance, \cite{SA,BK,XL, Eh,RY}, etc and references therein.
	
	Controllability is one of the fundamental concepts in mathematical control theory and engineering. The theory of controllability for the infinite dimensional control systems has been studied extensively, see for example, \cite{VB,BJK,P,EZ}, etc and references therein. In the  infinite dimensional setting, two basic concepts of controllability, namely exact controllability and approximate controllability are to be distinguished and have wide range of applications. In general, the concept of exact controllability seldom holds for the infinite dimensional control systems (cf. \cite{M,TRR,TR,EZ}). Therefore, it is essential to explore a weaker notion of controllability, known as approximate controllability. 
	The approximate controllability refers that the system can steer into an arbitrary small neighborhood of the final state. In the past, study of the approximate controllability for deterministic or stochastic evolution control systems with impulses and delays via fixed point theorems have produced excellent results, and the details can be found in various articles, see for example, \cite{SA,AB,RKG,M}, etc.
	
	There are many works reported on the approximate controllability of differential inclusions in Hilbert spaces, see for instance, \cite{AGK,ZL,VO,KR} etc. Grudzka et. al. \cite{AGK}, established sufficient conditions for the approximate controllability of the impulsive functional differential inclusions in Hilbert spaces. In \cite{ZL}, Liu et.al. examined the approximate controllability for control problems driven by a class of nonlinear evolution hemivariational inequalities in Hilbert spaces. Rykaczewski in \cite{KR}, investigated the approximate controllability of differential inclusions in Hilbert spaces via resolvent operator condition and fixed point theorem. 
	
	Recently, a few works \cite{AD,VV,QM,MY}, etc claimed the approximate controllability of  differential inclusions in Banach spaces via resolvent operator condition. But the resolvent operator defined in these works is valid only if the state space is a Hilbert space (see the discussions after the resolvent operator definition \eqref{opt:4.2}).   Moreover, some works examined the existence and controllability problems of the impulsive functional differential equations and inclusions with infinite delay, see for instance \cite{YKC,Eh,Er,VV,RY}, etc. The uniform norm is used to characterize the phase space in the case of non-impulsive systems or inclusions (cf. \cite{FUX,Ga}, etc).  The norm or seminorm considered for the phase space in the works \cite{YKC,VV}, etc (see application sections of \cite{Eh,Er,RY}, etc also) involves uniform norm (which is same as the non-impulsive case). However, the choice of such a norm or seminorm for the phase space in the impulsive case is not appropriate, for counter examples and more details, we refer \cite{GU}. Also the lower bound axiom appearing in the characterization of the phase space  holds true for  uniform norm, but it may not hold in the case of norms involving integrals (see \eqref{22} and \eqref{23} below, cf. \cite{VOj}). But in the work \cite{JRg} (see Chapter 3, Section 1, \cite{JRg}), the authors claim that it holds true for  norms or seminorms involving integrals, which may not be true in general. Furthermore, some of the works reported on differential inclusions, see for instance \cite{NIM,VV} etc,  claim that the operator \begin{align}\label{11}(\mathrm{G}h)(\cdot):= \int_{0}^{\cdot}\mathrm{U}(\cdot,s)h(s)\mathrm{d} s : \mathrm{L}^{1}([0,T];\mathbb{Y}) \rightarrow \mathrm{C}([0,T];\mathbb{Y}),\end{align} where $\mathbb{Y}$ is a Banach space,  is compact. This claim is not true in general (cf. Example 3.4, Chapter 3, \cite{LY}). 
	
	Motivated from the above discussions, in this work, we consider non-autonomous impulsive differential inclusions with infinite delay in Banach spaces. We first prove the existence of a mild solution of the controlled evolutionary inclusion problem \eqref{eqn:1.1} below and then establish sufficient conditions for the  approximate controllability of the same inclusion problem. Moreover, we have rectified the shortcomings observed in the works \cite{AD,VV,QM,MY}, etc and relaxed some assumptions on the nonlinear multi-valued map also. Our paper suitably modifies the phase space characterization  in the case of impulsive differential equations and inclusions with infinite delay by introducing the integral norm given in \eqref{norm} below, instead of uniform norm (also by avoiding the lower bound axiom). Furthermore, we also discuss the compactness of the operator defined in 	\eqref{11}, through Lemma \ref{lm3.2} for integrably bounded sequences in $\mathrm{L}^{1}([0,T];\mathbb{Y}) $.
	
	Let $\mathbb{X}$ be a separable reflexive Banach space  (having a strictly convex dual $\mathbb{X}^*$) and $\mathbb{U}$ be a separable Hilbert space. Let us consider the  following differential  inclusion:
	\begin{equation}\label{eqn:1.1}
	\left\{
	\begin{aligned}
	x'(t)&\in\mathrm{A}(t)x(t)+\mathrm{B}u(t)+\mathrm{F}(t,x_t), \ t\in J=[0,T], \\&\qquad\qquad  t\neq \tau_{k},\ k=1,\dots,m, \\
	\Delta x |_{t=\tau_{k}} &=I_{k}(x(\tau_{k})), \ k=1,\dots,m, \\
	x_{0}&=\phi \in \mathcal{B},
	\end{aligned}
	\right.
	\end{equation}
	where $ \{\mathrm{A}(t): t\in J\} $ is a family of linear operators (not necessarily bounded) on $\mathbb{X}$,   $\mathrm{B}: \mathbb{U} \rightarrow \mathbb{X}$ is a bounded linear operator, and the control function $u\in \mathrm{L}^{2}(J;\mathbb{U}).$ The impulses $ I_{k}: \mathbb{X} \rightarrow \mathbb{X}$ and $ \Delta x|_{t=\tau_{k}}=x(\tau_{k}^{+})-x(\tau_{k}^{-})$ with $0 = \tau_{0}<\tau_{1}< \dots \tau_{m}<\tau_{m+1}=T ,\;$ for $k=1,\ldots,m.$  The function $x_{t}:(-\infty, 0]\rightarrow\mathbb{X}$, $x_{t}(\theta)=x(t+\theta),$ belong to some phase space $\mathcal{B}$, which we will be specified in the next section. The nonlinear part $\mathrm{F}:J\times\mathcal{B}\to2^{\mathbb{X}}$ is a multi-valued function.
	
	The rest of the manuscript is structured as follows: Section \ref{pre} provides some fundamental definitions and preliminary results, which is useful in the later sections. In section \ref{sec3}, first we discuss the properties of the multi-valued  Nemytskii operator defined in \eqref{32} (Lemma \ref{lm3.1} and Lemma \ref{thm:3.1}). Next we prove an important result regarding  integrably bounded sequences (Lemma \ref{lm3.2} and Corollary \ref{cor:3.1}). In the same section, we prove one of our main results, which is about the existence of a mild solution to our problem \eqref{eqn:1.1} (Theorem \ref{thm3.2}). Section \ref{sec4} is devoted to the approximate controllability of the  non-autonomous impulsive differential inclusion with infinite delay given in \eqref{eqn:1.1}. In order to do this,  we first formulate the linear-quadratic regulator problem, obtain the existence of an optimal control (Theorem \ref{optimal}), and then we find the optimal control in the the feedback form (Lemma \ref{lem3.1}). Further, we provide Theorem \ref{thm4.2} to obtain the approximate controllability of the  single-valued non-autonomous linear control system \eqref{linear}. In the same section, we prove our main result on the approximate controllability of the differential  inclusion \eqref{eqn:1.1} and obtain sufficient conditions by using the resolvent operator condition. Finally, in section \ref{sec5}, we discuss a concrete example to demonstrate the application of the developed theory.  
	\section{Preliminaries}\label{pre}\setcounter{equation}{0}
As discussed in the introduction, $\mathbb{X}$ is a separable reflexive Banach space provided with the norm $\left\|\cdot\right\|_{\mathbb{X}},\ \mathbb{X}^*$ denotes its topological dual under the norm $\left\|\cdot\right\|_{\mathbb{X}^*}$ and the duality pairing between $\mathbb{X}$ and $\mathbb{X}^*$ is denoted by $\langle\cdot, \cdot\rangle$. Let $\mathbb{U}$ be a separable Hilbert space (identified with its own dual) equipped with the norm $\left\|\cdot\right\|_{\mathbb{U}}$ and the inner product $(\cdot, \cdot )$. The notation $\mathcal{L}(\mathbb{U},\mathbb{X})$ stands for the space of all bounded linear operators from $\mathbb{U}$ into $\mathbb{X}$ endowed with the operator norm $\left\|\cdot\right\|_{\mathcal{L}(\mathbb{U},\mathbb{X})}$. The space of all bounded linear operators defined on $\mathbb{X}$  is denoted by  $\mathcal{L}(\mathbb{X})$ with  the operator norm $\left\|\cdot\right\|_{\mathcal{L}(\mathbb{X})}$. By $\mathrm{PC}(J;\mathbb{X})$, we mean the space of all functions $x:J \rightarrow \mathbb{X}$ such that $x(\cdot)$ is continuous for $t\in J\backslash \{\tau_1,\ldots,\tau_m\}$ along with  $ x(\tau_{k}^{-})=x(\tau_k)$ and $x(\tau_{k}^{+})$ exists for all $k=1,\ldots,m,$ equipped with the norm $\left\|x\right\|_{\mathrm{PC}}= \sup\limits_{s\in J}\left\|x(s)\right\|_{\mathbb{X}}$.

	\subsection{Multi-valued map}
		Let us recall some preliminaries concerning multi-valued maps. For more detail we refer to \cite{Kd,Zd}. Let us given two topological vector spaces $Y$ and $Z$, a multi-valued map (or multifunction) $\Phi:Y\to 2^Z \backslash\{\varnothing\}$ (which is also denoted as $\Phi:Y\multimap Z$) is closed (convex) valued if $\Phi(y)$ is closed (convex) for all $y\in Y$. The multi-valued map $\Phi$ is upper semicontinuous (or u.s.c for short) if and only if $\Phi^{-1}(V):=\{y\in Y: \Phi(y)\subset V\}$ is an open subset of $Y$, for every open set $V\subset Z$. The multi-valued map $\Phi$ is  upper hemicontinuous (or u.h.c for short) if and only if $\Phi^{-1}(V)$ is a weakly open subset of $Y$ (that is, it is open in $Y_{w}$, more precisely, the space $Y$ equipped with weak topology), for every open set $V\subset Z$. If the multi-valued map $\Phi$ maps every bounded subset of $Y$ into a relatively compact set, then $\Phi$ is completely continuous. Moreover, a subset $V\subset Y$ is said to be weakly compact if it is compact in weak topology. The set $V$ is said to be relatively weakly compact if the closure of ${V}$ (in weak sense) is weakly compact.
	\begin{rem}\label{rem2.1}
			If a multi-valued map $\Phi:Y\multimap Z$ is completely continuous with nonempty compact values, then $\Phi$ is u.s.c. if and only if $\Phi$ has a closed graph, see \cite{Kd}.
	\end{rem}
	
		A closed valued multi-function $\Phi$  from an interval $[a, b]$ (equipped with Lebesgue $\sigma$-field) to a Banach space  $\mathbb{Y}$ is measurable, if the function given by $[a, b]\ni t\to d(x, \Phi(t)) := \inf \{\|x-z\|_{\mathbb{Y}}:z\in \Phi(t)\} $ is measurable for all $x\in\mathbb{Y}$ . A multi-valued mapping $\Phi: [a, b] \multimap\mathbb{Y}$ is integrably bounded, if there exists a function $\sigma\in\mathrm{L}^1([a,b];[0, +\infty))$ such that $$\left\|\Phi(t)\right\|_{\mathbb{Y}}:=\sup_{z\in\Phi(t)}\left\|z\right\|_{\mathbb{Y}}\le\sigma(t),\ \text{for a.e.}\ t\in[a,b].$$ A family $\mathcal{W}\subset\mathrm{L}^1([a,b];\mathbb{Y})$ is integrably bounded if $w:[a,b]\multimap\mathbb{Y}$, given by $w(t)=\{z(t):z\in\mathcal{W}\}$ is integrably bounded.
		
		By $S^1_{\Phi}$, we mean the set of all integrable selectors of the multi-valued map $\Phi:[a,b]\multimap\mathbb{Y}$, that is, $S^1_{\Phi}=\left\{f\in\mathrm{L}^1([a,b];\mathbb{Y}): f(t)\in\Phi(t), \ \text{a.e.}\right\}$. The set $S^1_{\Phi}$ may be empty, but if $\Phi$ is measurable and integrably bounded, then $S^1_{\Phi}\ne\varnothing$ (see, Section 2, Chapter 2, \cite{SH}). 
	We will use the following fixed point theorem to prove the existence results.
	\begin{theorem}(\cite{AG})\label{thm:2.1} Let $\mathbb{Y}$ be a Banach space, $\mathcal{K}\subset\mathbb{Y}$ be  nonempty, closed and convex with $0\in\mathcal{K}$, and  $\Phi:\mathcal{K}\multimap\mathcal{K}$ is  u.s.c with compact and convex valued, which maps each bounded set into a relatively compact set, then one of the following statements is true.
		\begin{itemize}
			\item [(i)] The set $\mathcal{D}:=\{x\in\mathcal{K}\ \vert\ x\in\kappa\Phi(x),\ \kappa\in(0,1)\}$ is unbounded.
			\item [(ii)] $\Phi$ has a fixed point, that is, there exists $x\in\mathcal{K}$ such that $x\in\Phi(x)$.
		\end{itemize}	
	\end{theorem}
	\subsection{Evolution family} 
	The theory of evolution family is an important tool to study the existence of solution and approximate controllability of non-autonomous systems. Let us first provide the definition of evolution family. 
	\begin{Def}[\cite{P}]{\label{def2.1}} A family of bounded linear operators $\left\{\mathrm{U}(t,s):0\leq s\leq t\leq T\right\}$ is said to be an \emph{evolution family}, if   
		\begin{enumerate}
			\item [(1)] $\mathrm{U}(s,s)=\mathrm{I}$, $\mathrm{U}(t,r)\mathrm{U}(r,s)=\mathrm{U}(t,s),$ for $0\leq s\leq r \leq t \leq T$.
			\item [(2)] The mapping $(t,s)\mapsto \mathrm{U}(t,s)$ is strongly continuous, for $0\leq s\leq t\leq T$.
		\end{enumerate}
	\end{Def}
	To construct an evolution family, let us impose the following assumptions on the family of linear operators $\left\{\mathrm{A}(t): t\in J \right\}$  (see, Chapter 5, \cite{P}). 
	\begin{Assumption}\label{ass2.1}
		\begin{enumerate}
			\item [\textit{(R1)}] The linear operator $\mathrm{A}(t)$ is closed for each $t\in J$ and the domain $\mathrm{D}(\mathrm{A}(t))=\mathrm{D}(\mathrm{A})$ is dense in $\mathbb{X}$ and  independent of $t$. 
			\item [\textit{(R2)}] The resolvent operator $\mathrm{R}(\lambda,\mathrm{A}(t))$ for $t\in J $ exists for all $\lambda$ with $\text{Re }\lambda \geq 0$ and there exists $K> 0$ such that 
			\begin{align*}
			\left\|\mathrm{R}(\lambda,\mathrm{A}(t))\right\|_{\mathcal{L}(\mathbb{X})}\leq \frac{K}{\left|\lambda\right|+1}.
			\end{align*}
			\item [\textit{(R3)}] There exist constants $N>0$ and $0<\delta \leq 1$ such that 
			\begin{align*}
			\left\|(\mathrm{A}(t)-\mathrm{A}(s))\mathrm{A}^{-1}(\tau)\right\|_{\mathcal{L}(\mathbb{X})}\leq N\left|t-s\right|^{\delta}, \ \text{for all}\ t,s,\tau \in J.
			\end{align*}
			\item [\textit{(R4)}] The operator $\mathrm{R}(\lambda,\mathrm{A}(t)), t\in J$ is compact for some $\lambda \in \rho(\mathrm{A}(t))$, where $\rho(\mathrm{A}(t))$ is the resolvent set of $\mathrm{A}(t)$. 
		\end{enumerate}
	\end{Assumption}
	\begin{lem}[Theorem 6.1, Chapter 5, \cite{P}]\label{lem2.1}
		Suppose that  (R1)-(R3) hold true. Then there exists a \emph{unique evolution family} $\mathrm{U}(t,s)$ on $0\leq s\leq t\leq T$ satisfying:
		\begin{enumerate}
			\item [(1)] There exists a constant $C\geq1$ such that $\left\|\mathrm{U}(t,s)\right\|_{\mathcal{L}(\mathbb{X})}\leq C$, for $0\leq s\leq t\leq T$.
			\item [(2)] The operator $\mathrm{U}(t,s):\mathbb{X}\rightarrow \mathrm{D}(\mathrm{A})$  for $0\leq s\leq t\leq T$ and the mapping $t\mapsto \mathrm{U}(t,s)$ is strongly differentiable in $\mathbb{X}$. The derivative $\frac{\partial}{\partial t}\mathrm{U}(t,s)\in\mathcal{L}(\mathbb{X})$ and it is strongly continuous on $0\leq s\leq t\leq T$. Moreover,
			\begin{align*}
			\frac{\partial}{\partial t}\mathrm{U}(t,s)-\mathrm{A}(t)\mathrm{U}(t,s)=0, \ \text{ for } \ 0\leq s< t\leq T,
			\end{align*}
			\begin{align*}
			\left\|\frac{\partial}{\partial t}\mathrm{U}(t,s)\right\|_{\mathcal{L}(\mathbb{X})}=\left\|\mathrm{A}(t)\mathrm{U}(t,s)\right\|_{\mathcal{L}(\mathbb{X})}\leq \frac{C}{t-s},
			\end{align*}
			and
			\begin{align*}
			\left\|\mathrm{A}(t)\mathrm{U}(t,s)\mathrm{A}(s)^{-1}\right\|_{\mathcal{L}(\mathbb{X})}\leq C, \ \text{ for }\ 0\leq s\leq t\leq T.
			\end{align*}
			\item [(3)] For each $t\in J$ and every $v\in \mathrm{D}(\mathrm{A})$, $\mathrm{U}(t,s)v$ is differentiable with respect to $s$ on $0\leq s\leq t\leq T$ and
			\begin{align*}
			\frac{\partial}{\partial s}\mathrm{U}(t,s)v=-\mathrm{U}(t,s)\mathrm{A}(s)v.
			\end{align*}
		\end{enumerate}
	\end{lem}
	\begin{lem}[Proposition 2.1, \cite{WF}]\label{lem2.2}
		Suppose $\left\{\mathrm{A}(t):t\in J\right\}$ satisfies the Assumptions (R1)-(R4). Let $\left\{\mathrm{U}(t,s):0\leq s\leq t\leq T\right\}$ be the linear evolution family generated by $\left\{\mathrm{A}(t):t\in J\right\}$, then $\left\{\mathrm{U}(t,s):0\leq s\leq t\leq T\right\}$ is \emph{a compact operator}, whenever $t-s> 0$.
	\end{lem}
	\subsection{Phase space}
	We now provide an axiomatic definition for the phase space $\mathcal{B},$ introduced by Hale and Kato in \cite{HY} and suitably modified to deal with the impulsive functional differential inclusion (cf. \cite{VOj}).  The space $\mathcal{B}$ is a linear space of all functions from $(-\infty, 0]$ into $\mathbb{X}$ endowed with the seminorm $\left\|\cdot\right\|_{\mathcal{B}}$ and satisfying the following axioms:
	\begin{enumerate}
		\item [(A1)] If $x: (-\infty, T]\rightarrow \mathbb{X},\;,$  such that $x_{0}\in \mathcal{B}$ and $x|_{J}\in \mathrm{PC}(J;\mathbb{X})$. Then, the following conditions hold:
		\begin{itemize}
			\item [(i)] $x_{t}\in\mathcal{B}$ for $t\in J$.
			\item [(ii)] $\left\|x_{t}\right\|_{\mathcal{B}}\leq \Lambda(t)\sup\{\left\|x(s)\right\|_{\mathbb{X}}: 0
			\leq s\leq t\}+\Upsilon(t)\left\|x_{0}\right\|_{\mathcal{B}},$  for $t\in J$, where $\Lambda, \Upsilon:[0, \infty)\rightarrow [0, \infty)$ are independent of $y$, the function $\Lambda(\cdot)$ is strictly positive and continuous, $\Upsilon(\cdot)$ is locally bounded.
		\end{itemize}
		\item [(A2)] The space $\mathcal{B}$ is complete. 
	\end{enumerate}  
	Let us consider the following examples of phase space.
	
\begin{Ex}\label{ex}For any $r>0$, we define a set 
	\begin{align*}
	\mathcal{P}:=\left\{\psi:[-r, 0]\to\mathbb{X}:\psi \ \text{is bounded and measurable }\right\}
	\end{align*} 
	equipped with the norm 
	\begin{align}
	\label{norm}\left\|\psi\right\|_{[-r, 0]}=\int_{-r}^{0}\left\|\psi(\theta)\right\|_{\mathbb{X}}\mathrm{d}\theta, \ \text{for all}\ \psi\in\mathcal{P}.
	\end{align}
	We now define a phase space $\mathcal{B}=\mathcal{B}_g$ as
	\begin{align*}
	\mathcal{B}_g:=\left\{\psi:(-\infty, 0]\to\mathbb{X}: \psi\vert_{[-r, 0]}\in\mathcal{P}\ \text{and}\ \int_{-\infty}^{0}g(s)\left\|\psi\right\|_{[s, 0]}\mathrm{d}s<+\infty \right\},
	\end{align*}
	endowed with the norm \begin{align}\label{22}\left\|\psi\right\|_{\mathcal{B}_g}=\int_{-\infty}^{0}g(s)\left\|\psi\right\|_{[s, 0]}\mathrm{d}s, \ \text{for all}\ \psi\in\mathcal{B}_g,\end{align}
	where $g:(-\infty, 0]\to(0, \infty)$ is a continuous function such that $\int_{-\infty}^{0}g(s)\mathrm{d}s=l<+\infty$. It is easy to verify that $(\mathcal{B}_g, \left\|\cdot\right\|_{\mathcal{B}_g}) $ is a Banach space.
\end{Ex}

\begin{Ex}
	For any $r>0$, let us define a set 
	\begin{align*}
	\mathcal{P'}:&=\{\psi:[-r, 0]\to\mathbb{X}:\psi \ \text{is piecewise continuous with a finite number of} \\&\qquad \text{discontinuity points}\ \{t_*\}\ \text{on}\ [-r,0)\ \ \mbox{such that}\ \psi(t_*^-) \ \mbox{and}\ \psi(t_*^+)\ \mbox{exist}.\}  
	\end{align*} 
	equipped with the norm defined in \eqref{norm}. We consider a phase space $\mathcal{B}=\mathcal{B}_g$ as
	\begin{align*}
	\mathcal{B}_g:=\left\{\psi:(-\infty, 0]\to\mathbb{X}: \psi\vert_{[-r, 0]}\in\mathcal{P'}\ \text{and}\ \int_{-\infty}^{0}g(s)\left\|\psi\right\|_{\mathbb{X}}\mathrm{d}s<+\infty \right\},
	\end{align*}
	endowed with the norm \begin{align}\label{23}\left\|\psi\right\|_{\mathcal{B}_g}=\int_{-\infty}^{0}g(s)\left\|\psi(s)\right\|_{\mathbb{X}}\mathrm{d}s, \ \text{for all}\ \psi\in\mathcal{B}_g,\end{align}
	where the function $g$ is same as previous example.	
\end{Ex} 	
Note that the uniform norm defined on the set $\mathcal{P}$ in the  definition of phase space given in \cite{YKC} is replaced by the integral norm given in \ref{norm}. The need for introducing such a norm in the phase space $\mathcal{B}_g$ is discussed later in Remark \ref{rem3.1}. 

We now define a set
\begin{align*} 
\mathrm{PC}_{\mathcal{B}}&:=\{x:(-\infty, T] \rightarrow \mathbb{X} : x_0=\phi\in\mathcal{B}\ \mbox{and}\ x\vert_{J}\in \mathrm{PC}(J;\mathbb{X})\}, \end{align*} with the norm 
$$\left\|x\right\|_{\mathrm{PC}_{\mathcal{B}}}=\left\|\phi\right\|_{\mathcal{B}}+\sup\limits_{t\in J} \left\|x(t)\right\|_{\mathbb{X}}.$$
Let us now verify the Assumption (A1) for the Example \ref{ex}.
	\begin{lem}\label{lm2.3}
		Suppose $x\in \mathrm{PC}_{\mathcal{B}_g}$, then for each $t\in J,\ x_t\in\mathcal{B}_g$. Moreover,
		\begin{align*}
		\left\|x_t\right\|_{\mathcal{B}_g}\leq\left\|\phi\right\|_{\mathcal{B}_g}+lt\sup_{0\le\theta\le t}\left\|x(\theta)\right\|_{\mathbb{X}},
		\end{align*}
		where $l=\int_{-\infty}^{0}g(s)\mathrm{d}s$. 
	\end{lem} 
	\begin{proof}
			For any $r>0$ and $t\in[0,r]$, it is easy to verify that the function $x_t$ is bounded and measurable on $[-r,0]$. Moreover, for $t\in J$, we estimate 
			\begin{align}
			\left\|x_t\right\|_{\mathcal{B}_g}&=\int_{-\infty}^{0}g(s)\left\|x_t\right\|_{[s,0]}\mathrm{d}s=\int_{-\infty}^{-t}g(s)\left\|x_t\right\|_{[s,0]}\mathrm{d}s+\int_{-t}^{0}g(s)\left\|x_t\right\|_{[s,0]}\mathrm{d}s\nonumber\\&=\int_{-\infty}^{-t}g(s)\left(\int_{s+t}^{t}\left\|x(\theta)\right\|_{\mathbb{X}}\mathrm{d}\theta\right)\mathrm{d}s+\int_{-t}^{0}g(s)\left(\int_{s+t}^{t}\left\|x(\theta)\right\|_{\mathbb{X}}\mathrm{d}\theta\right)\mathrm{d}s\nonumber\\&\le\int_{-\infty}^{-t}g(s)\left(\int_{s+t}^{0}\left\|x(\theta)\right\|_{\mathbb{X}}\mathrm{d}\theta\right)\mathrm{d}s+\int_{-\infty}^{-t}g(s)\left(\int_{0}^{t}\left\|x(\theta)\right\|_{\mathbb{X}}\mathrm{d}\theta\right)\mathrm{d}s\nonumber\\&\quad+\int_{-t}^{0}g(s)\left(\int_{0}^{t}\left\|x(\theta)\right\|_{\mathbb{X}}\mathrm{d}\theta\right)\mathrm{d}s\nonumber\\&=\int_{-\infty}^{-t}g(s)\left(\int_{s+t}^{0}\left\|x(\theta)\right\|_{\mathbb{X}}\mathrm{d}\theta\right)\mathrm{d}s+\int_{-\infty}^{0}g(s)\left(\int_{0}^{t}\left\|x(\theta)\right\|_{\mathbb{X}}\mathrm{d}\theta\right)\mathrm{d}s\nonumber\\&\le\int_{-\infty}^{0}g(s)\left(\int_{s}^{0}\left\|x(\theta)\right\|_{\mathbb{X}}\mathrm{d}\theta\right)\mathrm{d}s+\int_{-\infty}^{0}g(s)\mathrm{d}s\left(\int_{0}^{t}\left\|x(\theta)\right\|_{\mathbb{X}}\mathrm{d}\theta\right)\nonumber\\&=\left\|x_0\right\|_{\mathcal{B}_g}+l\int_{0}^{t}\left\|x(\theta)\right\|_{\mathbb{X}}\mathrm{d}\theta\nonumber\\&\le\left\|\phi\right\|_{\mathcal{B}_g}+lt\sup_{0\le\theta\le t}\left\|x(\theta)\right\|_{\mathbb{X}}\nonumber.
			\end{align}
			Since $\phi\in\mathcal{B}_g$, it is immediate that $x_t\in\mathcal{B}_g$. 

	\end{proof}
	\section{Existence of mild solution}\label{sec3}\setcounter{equation}{0} 
	In this section, we examine the existence of a mild solution for the problem \eqref{eqn:1.1}. First, we provide a definition of the mild solution for the differential inclusion \eqref{eqn:1.1}. 
	\begin{Def}
	For a given $u\in\mathrm{L}^2(J;\mathbb{U})$, a function $x:(-\infty, T]\to\mathbb{X}$ is called a \emph{mild solution} of the differential inclusion \eqref{eqn:1.1}, if there exists a function  $f\in\mathrm{L}^1(J;\mathbb{X})$ such that $f\in \mathrm{F}(t, x_t)$, the function $x(\cdot)$ satisfying  $ x(t)=\phi(t) , \text{ on } t\in (-\infty,0]$ (where $\phi \in \mathcal{B}$), $\Delta x|_{t=\tau_{k}}=I_{k}(x(\tau_{k})), \ k=1,\dots,m,$ and the restriction of $ x(\cdot) $ on the intervals $ J_k$, $k=0,\dots,m$ is continuous and is given by
	\begin{align*}
	x(t)&= \mathrm{U}(t,0)\phi(0) +\int_{0}^{t}\mathrm{U}(t,s)[\mathrm{B}u(s)+f(s)]\mathrm{d}s+\sum_{0<\tau_{k}<t}\mathrm{U}(t,\tau_{k})I_{k}(x(\tau_{k})),
	\end{align*}
	where $J_{0}=[0, \tau_{1}]$, $J_{k}=(\tau_{k}$, $\tau_{k+1}]$, $k=1,\dots,m. $
	\end{Def}
	We impose the following assumptions to prove the existence of a mild solution of the differential inclusion \eqref{eqn:1.1}:
	\begin{Assumption}\label{as3.1}  
		
		\begin{enumerate}
			\item[\textit{(H1)}] The operator $\mathrm{B}: \mathbb{U}\rightarrow\mathbb{X}$ is bounded with $\left\|\mathrm{B}\right\|_{\mathcal{L}(\mathbb{U},\mathbb{X})}= M_B$.
			\item [\textit{(H2)}]  The multi-valued map $F$ is weakly compact and convex valued.
			\item [\textit{(H3)}] For a.e. $t\in J$, the multi-valued map $\mathrm{F}(t, \cdot):\mathcal{B}\multimap\mathbb{X}$ is u.h.c and for every $\psi\in\mathcal{B},$ the function $\mathrm{F}(\cdot, \psi):J\multimap\mathbb{X}$ has a measurable selection.  
			\item [\textit{(H4)}] For each $r>0,$ there exist a function $\gamma\in\mathrm{L}^1(J;[0, +\infty))$ such that  $$  \left\|\mathrm{F}(t, \psi)\right\|_{\mathbb{X}}=\sup_{z\in \mathrm{F}(t, \psi)}\left\|z\right\|_{\mathbb{X}}\leq\gamma(t),\ \text{for a.e.}\ t\in J,\ \psi\in\mathcal{B}.$$ 
			\item [\textit{(H5)}] The impulses $ I_{k}: \mathbb{X} \rightarrow \mathbb{X}$,\ $k=1,\dots,m$, are continuous and satisfy
			$$  \left\|I_{k}(x) \right\|_{\mathbb{X}} \leq d_{k},\ \text{for all} \ x \in \mathbb{X}, \ k=1,\dots,m.$$
		\end{enumerate}
	\end{Assumption}
	We now introduce the Nemytskii operator $\mathrm{N}_{\mathrm{F}}:\mathrm{PC}(J;\mathbb{X})\multimap\mathrm{L}^1(J;\mathbb{X})$ as
	\begin{align}\label{32}
	\mathrm{N}_{\mathrm{F}}(x)=\left\{f\in\mathrm{L}^1(J;\mathbb{X}):f(t)\in \mathrm{F}(t, \tilde{x}_t),\ \text{for a.e.}\ t\in J \right\},
	\end{align}
	where $ \tilde{x}:(-\infty,T]\to\mathbb{X}$ is such that  $\tilde{x}(t)=\phi(t),\ t\in(-\infty, 0]$ and $\tilde{x}(t)=x(t)$ on $J$. It is easy to verify that the operator $\mathrm{N}_\mathrm{F}(x)$ is convex for every $x\in \mathrm{PC}(J;\mathbb{X})$ as the multi-valued map $\mathrm{F}$ is convex valued.  
	\begin{rem}\label{rem3.1}
		In general, the Nemytskii operator need not be nonempty valued. It appears to the authors that without introducing a suitable norm on the phase space $\mathcal{B}$, this operator often has no measurable selection, for instance see, Examples 4.1 and 4.2, \cite{GU}.  Equivalently we say that the function $\mathrm{F}(\cdot, \psi):J\multimap\mathbb{X}$ has a measurable selection for every $\psi\in \mathcal{B}$ is not true in general. This problem is finally overcome by using the integral norm given in \eqref{norm} instead of the uniform norm in the definition of phase space introduced in \cite{YKC}.
	\end{rem}
	We now provide some important properties of the Nemytskii operator $\mathrm{N}_\mathrm{F}(x)$.
	\begin{lem}\label{lm3.1}
		Let $x\in \mathrm{PC}(J;\mathbb{X})$ and $\mathrm{F}:J\times\mathcal{B}\multimap\mathbb{X}$ satisfy (H2)-(H4). Then there exists at least one Bochner integrable selection  $w(\cdot)\in \mathrm{F}(\cdot, \tilde{x}_{(\cdot)})$. In other words, Nemytskii operator has nonempty values.
	\end{lem}
	\begin{proof}
		Since $x\in \mathrm{PC}(J;\mathbb{X})$, more precisely $x:J\to\mathbb{X}$ is a piecewise continuous function with finite points of discontinuity. Let us consider a sequence $\{x^n\}_{n\ge1}$ of step functions, which converge uniformly to $x$ on $J$. By using the Assumption \ref{as3.1} (H3), there exists a measurable selection $w^{n}(t)\in \mathrm{F}(t,\tilde{x^n}_{t} )$. Since $w^{n}(t)$ is measurable for all $t\in J,\ n\ge1$, so $\{w^n(t): n\ge1\}$ is measurable and hence the closure of convex-hull, $\overline{\text{conv}\{w^{n}(t):  n\ge 1\}}$ is measurable as well (Theorem III.9, \cite{CC}). It is clear that $\tilde{x^n}(t)\to\tilde{x}(t)\ \text{as}\ n\to\infty$, uniformly for $t\in(-\infty, T]$, and then by axiom (A1)(ii), we have  
		\begin{align*}
		\left\|\tilde{x_t^n}-\tilde{x}_t\right\|_{\mathcal{B}}&\le\Lambda(t) \sup_{\theta\in[0,t]}\left\|x^n(\theta)-x(\theta)\right\|_{\mathbb{X}}\le K\sup_{\theta\in[0,T]}\left\|x^n(\theta)-x(\theta)\right\|_{\mathbb{X}}\\&\to0,\ \text{as}\ n\to\infty,\ \mbox{ uniformly for $t\in J$.}
		\end{align*}
		where $K=\sup_{t\in J}\Lambda(t)$. Let us define $\mathrm{W}(t):=\left\{\tilde{x_t^n}: n\ge1 \right\},\ t\in J$.  Clearly, $\mathrm{W}(t)\subset\mathcal{B}$ for $t\in J$. But $\tilde{x_t^n}$ converges $\tilde{x}_t$ uniformly in $t\in J$, and thus, it is evident the set $\mathrm{W}(t)$ is relatively compact. By employing upper hemicontinuity of $\mathrm{F}(t, \cdot)$ and Krein-Smulian theorem (cf. \cite{Rw}), we obtain that the set $\overline{\text{conv}(\mathrm{F}(t, \mathrm{W}(t)))}$ is weakly compact. Therefore, by applying the Dunford-Pettis theorem, we can find a subsequence of $w_n$, still denotes by $w_n$, such that
		$$ w_n \xrightharpoonup{w}w \ \text{ in }\ \mathrm{L}^1(J;\mathbb{X}), \ \mbox{as}\ n\to\infty.$$ 
		Moreover, by using the convergence result given in Theorem 3.2.6, \cite{JPA}, we obtain that $w(t)\in \mathrm{F}(t, \tilde{x}_t)$ for almost all $t\in J$.
	\end{proof}
	\begin{theorem}\label{thm:3.1}
	The Nemytskii operator $\mathrm{N}_{\mathrm{F}}$ is sequentially u.h.c with weakly compact values.
	\end{theorem}
	\begin{proof}
		A proof of the above can be obtained from Theorem 2.1, \cite{KR}.
	\end{proof}
		In order to establish the main result of this section, that is, the existence of  a mild solution of the problem \eqref{eqn:1.1}, first we prove the following lemma. In the next lemma, $\mathbb{X}$ can be any Banach space. 
	\begin{lem}\label{lm3.2}Let us define the operator $\mathrm{G}:\mathrm{L}^1(J;\mathbb{X})\to \mathrm{PC}(J;\mathbb{X})$ such that 
		\begin{align}\label{eqn:3.2}
		(\mathrm{G}h)(t):=\int_{0}^{t}\mathrm{U}(t,s)h(s)\mathrm{d}s,\ t\in J.
		\end{align}
		If $\{h_n\}_{n\ge1
		}\subset\mathrm{L}^1(J;\mathbb{X})$ is any integrably bounded sequence, then the sequence $v_n:=\mathrm{G}(h_n)$ is relatively compact.
	\end{lem}
	\begin{proof}
			Since we know that the sequence $h_n$ is integrably bounded, then there exists a function $\sigma\in\mathrm{L}^1(J;[0,+\infty))$ such that $$\left\|h_n(t)\right\|_{\mathbb{X}}\le\sigma(t), \ \text{for a.e.}\ t\in J.$$ For $t=0$, it is trivial. Take $t>0$ and for any $\varepsilon>0$, we can choose a $\delta\in(0,t)$  such that
			\begin{align}\label{eqn:3.3}
			C\int_{t-\delta}^{t}\sigma(s)\mathrm{d}s<\varepsilon/2,
			\end{align}
			where $C$ is the constant appearing in Lemma \ref{lem2.1} (1). We now define
			\begin{align*}
			y_n(t)=\int_{0}^{t-\delta}\mathrm{U}(t,s)h_n(s)\mathrm{d}s,\  n\ge1,\ t\in(0, T].
			\end{align*}
			Using Definition \ref{def2.1} (1), we can also write it as 
			\begin{align*}
			y_n(t)=&\mathrm{U}(t, t-\delta)\int_{0}^{t-\delta}\mathrm{U}(t-\delta,s)h_n(s)\mathrm{d}s\\
			=&\mathrm{U}(t, t-\delta)y^\delta_n(t),
			\end{align*}
			where $y^\delta_n(t)=\int_{0}^{t-\delta}\mathrm{U}(t-\delta,s)h_n(s)\mathrm{d}s$, for $n\ge1,\ t\in(0, T]$. Then it is easy to see that the sequence $\{y^\delta_n(t)\}_{n\ge1}$ is bounded in $\mathbb{X}$ for $t\in(0, T]$. Hence, by using the compactness of $\mathrm{U}(t,s)$ for $t-s>0$, we can find a finite $ z_{i}$'s, for $i=1,\dots, p $ in $ \mathbb{X} $ such that 
			\begin{align}\label{eqn:3.4}
			\{y_n(t)\}_{n\ge1}\subset \bigcup_{i=1}^{p}B(z_{i}, \varepsilon/2).
			\end{align}
			Using the definition of $v_n$ and $y_n$, and \eqref{eqn:3.3},  we compute 
			\begin{align*}
			\left\|v_n(t)-y_n(t)\right\|_{\mathbb{X}}&=\left\|\int_{0}^{t}\mathrm{U}(t,s)h_n(s)\mathrm{d}s-\int_{0}^{t-\delta}\mathrm{U}(t,s)h_n(s)\mathrm{d}s\right\|_{\mathbb{X}}\nonumber\\&\le\left\|\int_{t-\delta}^{t}\mathrm{U}(t,s)h_n(s)\mathrm{d}s\right\|_{\mathbb{X}}\le C\int_{t-\delta}^{t}\sigma(s)\mathrm{d}s\le\varepsilon/2.
			\end{align*}
			Consequently, from \eqref{eqn:3.4}, we infer that 
			\begin{align*}
			\{v_n(t)\}_{n\ge1}\subset\bigcup_{i=1}^{p}B(z_{i}, \varepsilon).
			\end{align*} 
			Thus, for each $t\in J,$ the sequence $\{v_n(t)\}_{n\ge1}$ is relatively compact in $\mathbb{X}$. 
			
			Next, we show that the sequence $\{v_n \}_{n\ge1}$ is equicontinuous on $J$. Let us take  $0\le t_1\le t_2\le T$ and estimate
			\begin{align}\label{eqn:3.5}
			\left\|v_n(t_2)-v_n(t_1)\right\|_{\mathbb{X}}&\le\int_{0}^{t_1}\left\|\left[\mathrm{U}(t_{2},s)-\mathrm{U}(t_{1},s)\right]h_n(s)\right\|_{\mathbb{X}}\mathrm{d}s+\int_{t_1}^{t_2}\left\|\mathrm{U}(t_2,s)h_n(s)\right\|_{\mathbb{X}}\mathrm{d}s\nonumber\\&\le\int_{0}^{t_1}\left\|\mathrm{U}(t_{2},s)-\mathrm{U}(t_{1},s)\right\|_{\mathcal{L}(\mathbb{X})}\sigma(s)\mathrm{d}s+C\int_{t_1}^{t_2}\sigma(s)\mathrm{d}s.
			\end{align}
			If $t_1=0,$ then from the above expression, we deduce that 
			\begin{align*}
			\lim_{t_2\to0^+}\left\|v_n(t_2)-v_n(t_1)\right\|_{\mathbb{X}}=0.
			\end{align*} 
			For $0<\delta<t_1<T$, from \eqref{eqn:3.5}, we have
			\begin{align}\label{eqn:3.6}
			&	\left\|v_n(t_2)-v_n(t_1)\right\|_{\mathbb{X}}\nonumber\\&\le\int_{0}^{t_1-\delta}\left\|\mathrm{U}(t_{2},s)-\mathrm{U}(t_{1},s)\right\|_{\mathcal{L}(\mathbb{X})}\sigma(s)\mathrm{d}s+\int_{t_1-\delta}^{t_1}\left\|\mathrm{U}(t_{2},s)-\mathrm{U}(t_{1},s)\right\|_{\mathcal{L}(\mathbb{X})}\sigma(s)\mathrm{d}s\nonumber\\&\quad+C\int_{t_1}^{t_2}\sigma(s)\mathrm{d}s\nonumber\\&\le\sup_{s\in[0,t_1-\delta]}\left\|\mathrm{U}(t_{2},s)-\mathrm{U}(t_{1},s)\right\|_{\mathcal{L}(\mathbb{X})}\int_{0}^{t_1-\delta}\sigma(s)\mathrm{d}s+2C\int_{t_1-\delta}^{t_1}\sigma(s)\mathrm{d}s+C\int_{t_1}^{t_2}\sigma(s)\mathrm{d}s.\nonumber\\
			\end{align}
			From Lemma \ref{lem2.2}, we infer that the evolution family $\mathrm{U}(t,s)$ is compact for $t-s>0$, and hence $\mathrm{U}(t,s)$ is continuous in the uniform operator topology for $\delta\leq s<t\leq T$ (see Theorem 3.2, Chapter 2, \cite{P}). Therefore, by using the continuity of $\mathrm{U}(t,s)$  in the uniform operator topology and using the arbitrariness of $ \delta$, the right hand side of \eqref{eqn:3.6} converges to zero as $|t_{2}-t_{1}| \rightarrow 0$. Thus, the sequence $\{v_n\}$ is equicontinuous on $J$. Hence, the sequence $\{v_n\}$ is totally bounded, so relatively compact.
	\end{proof}
	\begin{cor}\label{cor:3.1}
		If $\{h_n\}_{n\ge1}\subset\mathrm{L}^1(J;\mathbb{X})$ is any integrably bounded sequence such that $$h_n\xrightharpoonup{w} h \ \text{in}\ \mathrm{L}^1(J;\mathbb{X}), \mbox{as}\ n\to\infty$$ then
		$$\mathrm{G}(h_n)\to\mathrm{G}(h)\ \text{in}\ \mathrm{PC}(J;\mathbb{X}), \ \mbox{as}\ n\to\infty.$$ 
	\end{cor}
	\begin{rem}
		The above result is not true in  general, for instance, see Example 3.4, Chapter 3, \cite{LY}. In this way, we are rectifying some of the works reported on differential inclusions (cf. \cite{NIM,VV}), which claims the compactness of the operator $\mathrm{G}:\mathrm{L}^1(J;\mathbb{X})\to \mathrm{PC}(J;\mathbb{X})$ defined in \eqref{eqn:3.2}. 
	\end{rem}
	Let us now state and prove the main result of this section.
	\begin{theorem}\label{thm3.2}
		If the  Assumptions (R1)-(R4) and (H1)-(H5) hold true, then for each $u\in\mathrm{L}^2(J;\mathbb{U})$, the system \eqref{eqn:1.1} has a mild solution on $J$.
	\end{theorem}
	\begin{proof}
		Let us define \begin{align}\label{set}\mathrm{E}&:=\{x\in \mathrm{PC}_{\mathcal{B}} : x(0)=\phi(0)\}\ \text{be the space endowed with the norm}\ \left\|\cdot\right\|_{\mathrm{PC}_{\mathcal{B}}}.\end{align} We now define a multi-valued operator $\Gamma : \mathrm{E}\multimap\mathrm{E}$ as	
		\begin{align*}
		\Gamma(x)&=\bigg\{z\in \mathrm{E}: z(t)=\mathrm{U}(t,0)\phi(0) +\int_{0}^{t}\mathrm{U}(t,s)[\mathrm{B}u(s)+f(s)]\mathrm{d}s\\&\qquad+\sum_{0<\tau_{k}<t}\mathrm{U}(t,\tau_{k})I_{k}(\tilde{x}(\tau_{k})),\ f\in \mathrm{N}_{\mathrm{F}}(x) \bigg\},
		\end{align*}
		has nonempty values. It is clear from the definition of the multi-valued operator $\Gamma$ that the problem of finding a mild solution of system $\eqref{eqn:1.1}$ is equivalent to finding a fixed point of the operator $\Gamma$. We divide the proof of the multi-valued operator $\Gamma$ has a fixed point in the following steps.
		\vskip 0.1in 
		\noindent\textbf{Step (1):} \emph{$\Gamma(x)$ is convex for each $x\in \mathrm{E}$}. Let us assume $z_1, z_2\in\Gamma(x),$ for $x\in \mathrm{E}$, then there exist $f_1,f_2\in\mathrm{N}_\mathrm{F}(x)$ such that, for each $t\in J$, we have
		\begin{align*}
		z_i(t)= \mathrm{U}(t,0)\phi(0)+\int_{0}^{t}\mathrm{U}(t,s)\left[\mathrm{B}u(s)+f_i(s)\right]\mathrm{d}s+\sum_{0<\tau_{k}<t}\mathrm{U}(t,\tau_{k})I_{k}(\tilde{x}(\tau_{k})),
		\end{align*}
		for $i=1,2.$ For any $\lambda\in[0,1]$ and $x\in \mathrm{E}$, we compute
		\begin{align*}
		&(\lambda z_1+(1-\lambda)z_2)(t)\\&=\mathrm{U}(t,0)\phi(0)+\int_{0}^{t}\mathrm{U}(t,s)\mathrm{B}u(s)\mathrm{d}s+\int_{0}^{t}\mathrm{U}(t,s)\left[\lambda f_1(s)+(1-\lambda)f_2(s)\right]\mathrm{d}s\\&\quad+\sum_{0<\tau_{k}<t}\mathrm{U}(t,\tau_{k})I_{k}(\tilde{x}(\tau_{k})).
		\end{align*}
		Since the operator $\mathrm{N}_\mathrm{F}(x)$ is convex, we obtain, $\lambda f_1(s)+(1-\lambda)f_2(s)\in\mathrm{N}_\mathrm{F}(x),$ for $s\in J$. Hence, $\lambda z_1+(1-\lambda)z_2\in\Gamma(x)$, more precisely the multi-valued operator $\Gamma$ is convex. 
		\vskip 0.1in 
		\noindent\textbf{Step (2):} \emph{$\Gamma$ is bounded}. That is, the image of any bounded set under the multi-valued operator $\Gamma$ is bounded in $\mathrm{E}$. For any $q>0$, we consider a set  $\mathrm{E}_q:=\{x\in\mathrm{E} : \left\|x\right\|_{\mathrm{PC}_{\mathcal{B}}}\le q\}$. Let us take $x\in \mathrm{E}_q$ and $z\in\Gamma(x)$, then there exists $f\in\mathrm{N}_\mathrm{F}(x)$ such that 
		\begin{align}\label{eqn:3.7}
		z(t)&=\mathrm{U}(t,0)\phi(0)+\int_{0}^{t}\mathrm{U}(t,s)\left[\mathrm{B}u(s)+f(s)\right]\mathrm{d}s+\sum_{0<\tau_{k}<t}\mathrm{U}(t,\tau_{k})I_{k}(\tilde{x}(\tau_{k})).
		\end{align}
		We now estimate
		\begin{align}
		&\left\|z(t)\right\|_{\mathbb{X}}\nonumber\\&\le\left\|\mathrm{U}(t,0)\phi(0)\right\|_{\mathbb{X}}+\int_{0}^{t}\left\|\mathrm{U}(t,s)\left[\mathrm{B}u(s)+f(s)\right]\right\|_{\mathbb{X}}\mathrm{d}s+\sum_{0<\tau_{k}<t}\left\|\mathrm{U}(t,\tau_{k})I_{k}(\tilde{x}(\tau_{k}))\right\|_{\mathbb{X}}\nonumber\\&\le \left\|\mathrm{U}(t,0)\right\|_{\mathcal{L}(\mathbb{X})}\left\|\phi(0)\right\|_{\mathbb{X}}+\int_{0}^{t}\left\|\mathrm{U}(t,s)\right\|_{\mathcal{L}(\mathbb{X})}\left[\left\|\mathrm{B}\right\|_{\mathcal{L}(\mathbb{U},\mathbb{X})}\left\|u(s)\right\|_{\mathbb{U}}+\left\|f(s)\right\|_{\mathbb{X}}\right]\mathrm{d}s\nonumber\\&\quad+\sum_{0<\tau_{k}<t}\left\|\mathrm{U}(t,\tau_{k})\right\|_{\mathcal{L}(\mathbb{X})}\left\|I_{k}(\tilde{x}(\tau_{k}))\right\|_{\mathbb{X}}\nonumber\\&\le C\left\|\phi(0)\right\|_{\mathbb{X}}+CM_B\left\|u\right\|_{\mathrm{L}^2(J;\mathbb{U})}\sqrt{T}+C\left\|\gamma\right\|_{\mathrm{L}^1(J;[0,+\infty))}+C\sum_{k=1}^{m}d_k,\nonumber
		\end{align}
		for all $t\in J$. Here, we have used the condition (1)  of Lemma \ref{lem2.1}, Assumption \ref{as3.1}, and H\"{o}lder inequality. Thus, the above expression ensures that the operator $\Gamma$ is bounded. 
		\vskip 0.1in 
		\noindent\textbf{Step (3):} \emph{The image $\Gamma(\mathrm{E}_q)$ is equicontinuous for every $q>0$}. Since we know that for $z\in\Gamma(x)$ with $x\in \mathrm{E}_q$, there exists $f\in\mathrm{N}_\mathrm{F}(x)$ such that \eqref{eqn:3.7} holds. Using Lemma \ref{lem2.1} and Assumptions \ref{as3.1}, for $0\le t_1\le t_2\le T$ and $x\in \mathrm{E}_q$, we compute 
		\begin{align*}
		&\left\|z(t_2)-z(t_1)\right\|_{\mathbb{X}}\nonumber\\&\le\left\|\left[\mathrm{U}(t_{2},0)-\mathrm{U}(t_{1},0)\right]\phi(0)\right\|_{\mathbb{X}}+\int_{0}^{t_1}\left\|\left[\mathrm{U}(t_{2},s)-\mathrm{U}(t_{1},s)\right]\left[\mathrm{B}u(s)+f(s)\right]\right\|_{\mathbb{X}}\mathrm{d}s\nonumber\\&\quad+\int_{t_1}^{t_2}\left\|\mathrm{U}(t_2,s)\left[\mathrm{B}u(s)+f(s)\right]\right\|_{\mathbb{X}}\mathrm{d}s+\sum_{0<\tau_{k}<t_{1}}\left\|\left[\mathrm{U}(t_{2},\tau_{k})-\mathrm{U}(t_{1},\tau_{k})\right]I_{k}(\tilde{x}(\tau_{k}))\right\|_{\mathbb{X}}\nonumber\\&\quad+\sum_{t_1\leq\tau_{k}\leq t_2}\left\|\mathrm{U}(t_{2},\tau_{k})I_{k}(\tilde{x}(\tau_{k}))\right\|_{\mathbb{X}}\nonumber\\&\le\left\|\left[\mathrm{U}(t_{2},0)-\mathrm{U}(t_{1},0)\right]\phi(0)\right\|_{\mathbb{X}}+\int_{0}^{t_1}\left\|\mathrm{U}(t_{2},s)-\mathrm{U}(t_{1},s)\right\|_{\mathcal{L}(\mathbb{X})}\left[M_B\left\|u(s)\right\|_{\mathbb{U}}+\gamma(s)\right]\mathrm{d}s\nonumber\\&\quad+C\int_{t_1}^{t_2}\left[M_B\left\|u(s)\right\|_{\mathbb{U}}+\gamma(s)\right]\mathrm{d}s+\sum_{0<\tau_{k}<t_{1}}\left\|\mathrm{U}(t_{2},\tau_{k})-\mathrm{U}(t_{1},\tau_{k})\right\|_{\mathcal{L}(\mathbb{X})}d_k\nonumber\\&\quad+C\sum_{t_1\leq\tau_{k}\leq t_2}d_k.
		\end{align*}
		If $t_1=0,$ then from the above expression, we deduce that 
		$$\lim_{t_2\to0^+}\left\|z(t_2)-z(t_1)\right\|_{\mathbb{X}}=0,\; \text{ unifromly for } x\in \mathrm{E}_q.$$
		For $0<\delta<t_1<T$, we have
		\begin{align*}
		&	\left\|z(t_2)-z(t_1)\right\|_{\mathbb{X}}\nonumber\\&\le\left\|\mathrm{U}(t_{2},0)-\mathrm{U}(t_{1},0)\right\|_{\mathcal{L}(\mathbb{X})}\left\|\phi(0)\right\|_{\mathbb{X}}+C\int_{t_1}^{t_2}\left[M_B\left\|u(s)\right\|_{\mathbb{U}}+\gamma(s)\right]\mathrm{d}s\nonumber\\&\quad+\int_{0}^{t_1-\delta}\left\|\mathrm{U}(t_{2},s)-\mathrm{U}(t_{1},s)\right\|_{\mathcal{L}(\mathbb{X})}\left[M_B\left\|u(s)\right\|_{\mathbb{U}}+\gamma(s)\right]\mathrm{d}s\nonumber\\&\quad+\int_{t_1-\delta}^{t_1}\left\|\mathrm{U}(t_{2},s)-\mathrm{U}(t_{1},s)\right\|_{\mathcal{L}(\mathbb{X})}\left[M_B\left\|u(s)\right\|_{\mathbb{U}}+\gamma(s)\right]\mathrm{d}s\nonumber\\&\quad+\sum_{0<\tau_{k}<t_{1}}\left\|\mathrm{U}(t_{2},\tau_{k})-\mathrm{U}(t_{1},\tau_{k})\right\|_{\mathcal{L}(\mathbb{X})}d_k+C\sum_{t_1\leq\tau_{k}\leq t_2}d_k\nonumber\\&\le\left\|\mathrm{U}(t_{2},0)-\mathrm{U}(t_{1},0)\right\|_{\mathcal{L}(\mathbb{X})}\left\|\phi(0)\right\|_{\mathbb{X}}+CM_B\left\|u\right\|_{\mathrm{L}^2(J;\mathbb{U})}(t_2-t_1)^{1/2}+C\int_{t_1}^{t_2}\gamma(s)\mathrm{d}s\nonumber\\&\quad+\sup_{s\in[0,t_1-\delta]}\left\|\mathrm{U}(t_{2},s)-\mathrm{U}(t_{1},s)\right\|_{\mathcal{L}(\mathbb{X})}\left[M_B\left\|u\right\|_{\mathrm{L}^2(J;\mathbb{U})}\sqrt{T}+\int_{0}^{t_1-\delta}\gamma(s)\mathrm{d}s\right]\nonumber\\&\quad+2C\left[M_B\left\|u\right\|_{\mathrm{L}^2(J;\mathbb{U})}\sqrt{\delta}+\int_{t_1-\delta}^{t_1}\gamma(s)\mathrm{d}s\right]+\sum_{0<\tau_{k}<t_{1}}\left\|\mathrm{U}(t_{2},\tau_{k})-\mathrm{U}(t_{1},\tau_{k})\right\|_{\mathcal{L}(\mathbb{X})}d_k\nonumber\\&\quad+C\sum_{t_1\leq\tau_{k}\leq t_2}d_k.
		\end{align*}
		Similar to the estimate \eqref{eqn:3.6}, it is easy to verify that for arbitrary $\delta$, the right hand side of the above estimate converges to zero uniformly for $x\in \mathrm{E}_q,$ whenever $|t_{2}-t_{1}| \rightarrow 0$.
		\vskip 0.1in 
		\noindent\textbf{Step (4):} \emph{$\Gamma$ is completely continuous}. For this, we need to show that for any fixed $q>0,$ the image of $\mathrm{E}_q$ under the map $\Gamma$ is relatively compact in $\mathrm{E}$. In order to do this, taking into account of Steps 2 and 3 and the Arzela-Ascoli theorem, it  suffices to show that the set $\mathrm{V}(t):=\{z(t): z\in\Gamma(\mathrm{E}_q)\}$, for all $t\in J$ is relatively compact in $\mathbb{X}$. For $t=0$, it is trivial. Take  $0<t\le T$ and choose $0<\varepsilon<t$, then we  define the operator $\Gamma^\varepsilon:\mathrm{E}\multimap\mathrm{E}$ as
		\begin{align*}
		\Gamma^\varepsilon(x):&=\bigg\{z^\varepsilon\in\mathrm{E}: z^\varepsilon(t)= \mathrm{U}(t,0)\phi(0)+\int_{0}^{t-\varepsilon}\mathrm{U}(t,s)\left[\mathrm{B}u(s)+f(s)\right]\mathrm{d}s\nonumber\\&\qquad+\sum_{0<\tau_{k}<t-\varepsilon}\mathrm{U}(t,\tau_{k})I_{k}(\tilde{x}(\tau_{k})),\ f\in\mathrm{N}_{\mathrm{F}}(x)\bigg\}.
		\end{align*}
		For $\ x\in \mathrm{E}_q,$ and $z^\varepsilon\in\Gamma^\varepsilon(x)$, there exists $f\in\mathrm{N}_{\mathrm{F}}(x)$ such that
		\begin{align*}
		z^\varepsilon(t)&=\mathrm{U}(t,0)\phi(0)+\int_{0}^{t-\varepsilon}\mathrm{U}(t,s)\left[\mathrm{B}u(s)+f(s)\right]\mathrm{d}s+\sum_{0<\tau_{k}<t-\varepsilon}\mathrm{U}(t,\tau_{k})I_{k}(\tilde{x}(\tau_{k}))\nonumber\\&=\mathrm{U}(t,0)\phi(0)+\mathrm{U}(t,t-\varepsilon)\bigg[\int_{0}^{t-\varepsilon}\mathrm{U}(t-\varepsilon,s)\left[\mathrm{B}u(s)+f(s)\right]\mathrm{d}s\nonumber\\&\quad+\sum_{0<\tau_{k}<t-\varepsilon}\mathrm{U}(t-\varepsilon,\tau_{k})I_{k}(\tilde{x}(\tau_{k}))\bigg]\nonumber\\&=\mathrm{U}(t,t-\varepsilon)y^{\varepsilon}(t),
		\end{align*}
		where 
		\begin{align*}
		y^{\varepsilon}(t)&=\mathrm{U}(t-\varepsilon, 0)\phi(0)+\int_{0}^{t-\varepsilon}\mathrm{U}(t-\varepsilon,s)\left[\mathrm{B}u(s)+f(s)\right]\mathrm{d}s\\&\quad+\sum_{0<\tau_{k}<t-\varepsilon}\mathrm{U}(t-\varepsilon,\tau_{k})I_{k}(\tilde{x}(\tau_{k})).
		\end{align*}
		Using Lemma \ref{lem2.1} and Assumption \ref{as3.1}, one can estimate $\|y^{\varepsilon}(\cdot)\|_{\mathbb{X}}$ as
		\begin{align*}
		\|y^{\varepsilon}(t)\|_{\mathbb{X}}&\leq C\left\|\phi(0)\right\|_{\mathbb{X}}+CM_B\left\|u\right\|_{\mathrm{L}^2(J;\mathbb{U})}\sqrt{T}+C\left\|\gamma\right\|_{\mathrm{L}^1(J;[0, +\infty))}+C\sum_{k=1}^{m}d_k.
		\end{align*}
		Since the operator $ \mathrm{U}(t,s),$ for $t-s>0$ is compact,  the set $ \mathrm{V}_{\varepsilon}(t)=\{z^\varepsilon(t):z^\varepsilon\in\Gamma^{\varepsilon}(\mathrm{E}_q)\}$ is relatively compact in $ \mathbb{X}$. Moreover, for every $z\in\Gamma(\mathrm{E}_q)$, we deduce that
		\begin{align*}
		\nonumber\left\|z(t)-z^\varepsilon(t)\right\|_{\mathbb{X}}&\le\left\|\int_{0}^{t}\mathrm{U}(t,s)\left[\mathrm{B}u(s)+f(s)\right]\mathrm{d}s-\int_{0}^{t-\varepsilon}\mathrm{U}(t,s)\left[\mathrm{B}u(s)+f(s)\right]\mathrm{d}s\right\|_{\mathbb{X}}\\&\quad+\left\|\sum_{0<\tau_{k}<t}\mathrm{U}(t,\tau_{k})I_{k}(\tilde{x}(\tau_{k}))-\sum_{0<\tau_{k}<t-\varepsilon}\mathrm{U}(t,\tau_{k})I_{k}(\tilde{x}(\tau_{k}))\right\|_\mathbb{X}\\&\le\left\|\int_{t-\varepsilon}^{t}\mathrm{U}(t,s)\left[\mathrm{B}u(s)+f(s)\right]\mathrm{d}s\right\|_{\mathbb{X}}+\left\|\sum_{t-\varepsilon<\tau_{k}<t}\mathrm{U}(t,\tau_{k})I_{k}(\tilde{x}(\tau_{k}))\right\|_\mathbb{X}\\&\le C\left(M_B\left\|u\right\|_{\mathrm{L}^2(J;\mathbb{X})}\sqrt{\varepsilon}+\int_{t-\varepsilon}^{t}\gamma(s)\mathrm{d}s+\sum_{t-\varepsilon<\tau_k<t}d_k \right)\\&\to 0, \ \text{as}\  \varepsilon\to 0.
		\end{align*} 
		Thus, the set $\mathrm{V}(t)$ is totally bounded (following similarly as in the proof of Lemma \ref{lm3.2}), so by using  Step 3, we obtain that $\mathrm{V}(t)$ is relatively compact in $\mathbb{X}$.
		\vskip 0.1in 
		\noindent\textbf{Step (5):} \emph{$\Gamma$ is upper semicontinuous}. Since $\mathbb{X}$ is a reflexive Banach space,  in view of Step (4) and Remark \ref{rem2.1}, we infer that  the operator $\Gamma$ is u.s.c if and only if $\Gamma$ has a closed graph. To prove $\Gamma$ has a closed graph, we consider two sequences $x^n\to x^*$ in $\mathrm{E}$ and $z^n\to z^*$ in $\mathrm{E}$ with $z^n\in\Gamma(x^n)$. Since $z^n\in\Gamma(x^n)$, there exists $f_n\in \mathrm{N}_\mathrm{F}(x^n) $ such that 
		\begin{align}\label{eqn:3.13}
		z^n(t)= \mathrm{U}(t,0)\phi(0)+\int_{0}^{t}\mathrm{U}(t,s)\left[\mathrm{B}u(s)+f_n(s)\right]\mathrm{d}s+\sum_{0<\tau_{k}<t}\mathrm{U}(t,\tau_{k})I_{k}(\tilde{x}(\tau_{k})).
		\end{align}
		By using the Assumption \ref{as3.1} \textit{(H4)}, it is easy to verify that the sequence $\{f_n\}_{n\ge 1}$ is integrably bounded. Since the multi-valued map $\mathrm{F}$ is weakly compact valued and $f_n(t)\in \mathrm{F}(t, \tilde{x_t^n})$, then by using Dunford-Pettis theorem $\{f_n\}_{n\ge 1}$ is relatively weakly compact in $\mathrm{L}^1(J;\mathbb{X})$. Therefore, we can find a subsequence of $\{f_n\}_{n\ge 1}$, still denoted as $\{f_n\}_{n\ge 1}$ such that 
		\begin{align}\label{eqn:3.14}
		f_n\xrightharpoonup{w}  f^{*} \ \text{in}\ \mathrm{L}^1(J;\mathbb{X}), \ \mbox{as}\ n\to\infty. 
		\end{align}
		By passing $n\to\infty$ in \eqref{eqn:3.13}, we get
		\begin{align}\label{eqn:3.15}
		z^n(t)\to\mathrm{U}(t,0)\phi(0)+\int_{0}^{t}\!\!\!\mathrm{U}(t,s)\left[\mathrm{B}u(s)+f^{*}(s)\right]\mathrm{d}s+\sum_{0<\tau_{k}<t}\mathrm{U}(t,\tau_{k})I_{k}(\tilde{x}(\tau_{k}))=z^{0}(t),\end{align}
		where we used the weak convergence \eqref{eqn:3.14} together with Corollary \ref{cor:3.1}. Since we know that $x^n\to x^*$ in $\mathrm{E}$ and $f_n\in\mathrm{N}_\mathrm{F}(x^n)$, then by using the weak convergence of \eqref{eqn:3.14} and Lemma \ref{thm:3.1}, we obtain a subsequence  $f_{n_k}\xrightharpoonup{w} f^{*}\in\mathrm{N}_\mathrm{F}(x^*)$ as $k\to\infty$ (by the uniqueness of weak limit). This implies that  $z^0\in\Gamma(x^*)$. Moreover, the convergence \eqref{eqn:3.15} guarantees that $z^0(t)= z^*(t)$ for $t\in J$. Thus, we have $z^*\in\Gamma(x^*)$, which ensures that the graph of $\Gamma$ is closed.
		\vskip 0.1in 
		\noindent\textbf{Step (6):} \emph{The set $\mathcal{D}:=\{x\in\mathrm{E}: x\in\kappa\Gamma(x), \ 0<\kappa<1\}$ is bounded}. Let us take $x\in\mathcal{D}$ and $0<\kappa<1$. Then there exists $f\in\mathrm{N}_\mathrm{F}(x)$ such that
		\begin{align*}
		x(t)=\kappa\mathrm{U}(t,0)\phi(0)+\kappa\int_{0}^{t}\mathrm{U}(t,s)\left[\mathrm{B}u(s)+f(s)\right]\mathrm{d}s+\kappa\sum_{0<\tau_{k}<t}\mathrm{U}(t,\tau_{k})I_{k}(\tilde{x}(\tau_{k})).
		\end{align*}
		Using the condition (1) of Lemma \ref{lem2.1} and Assumption \ref{as3.1}, we estimate 
		\begin{align*}
		&\left\|x(t)\right\|_{\mathbb{X}}\\&\le\left\|\mathrm{U}(t,0)\phi(0)\right\|_{\mathbb{X}}+\int_{0}^{t}\left\|\mathrm{U}(t,s)\left[\mathrm{B}u(s)+f(s)\right]\right\|_{\mathbb{X}}\mathrm{d}s+\sum_{0<\tau_{k}<t}\left\|\mathrm{U}(t,\tau_{k})I_{k}(\tilde{x}(\tau_{k}))\right\|_{\mathbb{X}}\nonumber\\&\le \left\|\mathrm{U}(t,0)\right\|_{\mathcal{L}(\mathbb{X})}\left\|\phi(0)\right\|_{\mathbb{X}}+\int_{0}^{t}\left\|\mathrm{U}(t,0)\right\|_{\mathcal{L}(\mathbb{X})}\left[\left\|\mathrm{B}\right\|_{\mathcal{L}(\mathbb{U},\mathbb{X})}\left\|u(s)\right\|_{\mathbb{U}}+\left\|f(s)\right\|_{\mathbb{X}}\right]\mathrm{d}s\nonumber\\&\quad+\sum_{0<\tau_{k}<t}\left\|\mathrm{U}(t,\tau_{k})\right\|_{\mathcal{L}(\mathbb{X})}\left\|I_{k}(\tilde{x}(\tau_{k}))\right\|_{\mathbb{X}}\nonumber\\&\le C\left\|\phi(0)\right\|_{\mathbb{X}}+CM_B\left\|u\right\|_{\mathrm{L}^2(J;\mathbb{X})}\sqrt{T}+C\left\|\gamma\right\|_{\mathrm{L}^1(J;[0,+\infty))}+C\sum_{k=1}^{m}d_k,
		\end{align*}
		for all $t\in J.$ The above estimate ensures that the set $\mathcal{D}$ is bounded.
		
		Then, by invoking Theorem \ref{thm:2.1}, the operator $\Gamma$ has a fixed point, which implies that the system \eqref{eqn:1.1} has a mild solution.
	\end{proof}
	\section{Approximate Controllability Result for Semilinear System} \label{sec4}\setcounter{equation}{0}
	In this section, we investigate the approximate controllability of the system \eqref{eqn:1.1}. This will be done through the linear control problem corresponding to the system \eqref{eqn:1.1}. First, we define a resolvent operator which is an important and useful tool to study the approximate controllability of control systems. Let us define the following operators:
	\begin{equation}\label{opt:4.2}
	\left\{
	\begin{aligned}
	L_Tu&:=\int_0^T\mathrm{U}(T,t)\mathrm{B}u(t)\mathrm{d}t,\\
	\Psi_{0}^{T}&:=\int^{T}_{0}\mathrm{U}(T,t)\mathrm{B}\mathrm{B}^{*}\mathrm{U}^{*}(T,t)\mathrm{d}t=L_T(L_T)^*,\\
	\mathrm{R}(\lambda,\Psi_{0}^{T})&:=(\lambda \mathrm{I}+\Psi_{0}^{T}\mathcal{J})^{-1},\ \lambda > 0,
	\end{aligned}
	\right.
	\end{equation}
	where $\mathrm{B}^{*}$ and $\mathrm{U}^{*}(t,s)$ denote the adjoint operators of $\mathrm{B}$ and $\mathrm{U}(t,s),$ respectively. The map $\mathcal{J} : \mathbb{X} \multimap\mathbb{X}^*$ is called duality mapping which is defined as 
	\begin{align*}
	\mathcal{J}&=\{x^* \in \mathbb{X}^* : \langle x, x^* \rangle=\left\|x\right\|_{\mathbb{X}}^2= \left\|x^*\right\|_{\mathbb{X}^*}^2 \}, \text{ for all } x\in \mathbb{X}.
	\end{align*}
	If the space $\mathbb{X}$ is a reflexive Banach space, then $\mathbb{X}$ can be renormed such that $\mathbb{X}$ and $\mathbb{X}^*$ become strictly convex (\cite{AA}). From the strict convexity of $\mathbb{X}^*$,  the mapping $\mathcal{J}$ becomes single-valued as well as demicontinuous, that is, $$x_k\to x\ \text{ in }\ \mathbb{X}\ \text{ implies }\ \mathcal{J}[x_k] \xrightharpoonup{w} \mathcal{J}[x]\ \text{ in } \ \mathbb{X}^*\ \text{ as }\ k\to\infty.$$
	Note that if $\mathbb{X}$ is a separable Hilbert space, which is identified with its own dual (the duality mapping $\mathcal{J}$ becomes $\mathrm{I}$, the identity operator), then the resolvent operator is defined as $\mathrm{R}(\lambda,\Psi_{0}^{T}):=(\lambda \mathrm{I}+\Psi_{0}^{T})^{-1},\ \lambda > 0$. 
	
	We need the following assumption to establish the approximate controllability of the linear and nonlinear systems. 
	\begin{Assumption}\label{ass4.1}
	We assume that 
	\begin{enumerate}
		\item [\textit{(H0)}] for every $h\in\mathbb{X}$, $z_{\lambda}(h)=\lambda(\lambda \mathrm{I}+\Psi_{0}^{T}\mathcal{J})^{-1}(h) \rightarrow 0$ as $\lambda\downarrow 0$ in strong topology, where $z_{\lambda}(h)$ is a solution of the equation: \begin{align}\label{eqn:4.3}\lambda z_{\lambda}+\Psi_{0}^{T}\mathcal{J}[z_{\lambda}]=\lambda h.\end{align}
	\end{enumerate}	
	\end{Assumption}
	\begin{rem}
		If $\mathbb{X}$ is a separable reflexive Banach space, then for every $h\in\mathbb{X}$ and $\lambda>0$, the equation \eqref{eqn:4.3} has a unique solution  $z_{\lambda}(h)=\lambda(\lambda \mathrm{I}+\Psi_{0}^{T}\mathcal{J})^{-1}(h)=\lambda\mathrm{R}(\lambda,\Psi_{0}^{T})(h)$ (see Lemma 2.2, \cite{M}). Moreover,
		\begin{align}\label{eqn:4.4}
		\left\|z_{\lambda}(h)\right\|_{\mathbb{X}}=\left\|\mathcal{J}[z_{\lambda}(h)]\right\|_{\mathbb{X}^{*}}\leq\left\|h\right\|_{\mathbb{X}}.
		\end{align}	
	\end{rem}
	\begin{Def}
	The system given in (\ref{eqn:1.1}) is said to be \emph{approximately controllable} on $ J $, if $\ \overline{\mathfrak{R}(T;\phi,u)}=\mathbb{X}$, where $\mathfrak{R}(T;\phi,u)$ (\emph{reachable set})  is define as \begin{align*}\mathfrak{R}(T;\phi,u) = \{x(T;\phi,u): u\in \mathrm{L}^{2}(J;\mathbb{U})\}.\end{align*}
	\end{Def}
	\subsection{Linear control problem}In this subsection, we investigate the approximate controllability of the linear control problem corresponding to \eqref{eqn:1.1} in terms of the resolvent operator. To establish this, we first formulate an optimal control problem by considering the linear-quadratic regulator problem, consisting of minimizing a cost functional. The cost functional is given  by  
	\begin{equation}\label{eqn:4.5}
	\mathcal{F}(x,u)=\left\|x(T)-x_{T}\right\|^{2}_{\mathbb{X}}+\lambda\int^{T}_{0}\left\|u(t)\right\|^{2}_{\mathbb{U}}\mathrm{d}t,
	\end{equation}
	where $x(\cdot)$ is the solution of the linear control system 
	\begin{equation}\label{linear}
	\left\{
	\begin{aligned}
	x'(t)&= \mathrm{A}(t)x(t)+\mathrm{B}u(t),\ t\in J,\\
	x(0)&=\phi(0), \ \phi\in\mathcal{B},
	\end{aligned}
	\right.
	\end{equation}
	with the control $u\in \mathbb{U}$, $x_{T}\in \mathbb{X}$ and $\lambda >0$. Since $\mathrm{B}u\in\mathrm{L}^1(J;\mathbb{X})$, the system \eqref{linear} has a unique mild solution $x\in \mathrm{C}(J;\mathbb{X}) $ given by ( Corollary 2.2, Chapter 4, \cite{P})
	\begin{align*}
	x(t)= \mathrm{U}(t,0)\phi(0)+\int^{t}_{0}\mathrm{U}(t,s)\mathrm{B}u(s)\mathrm{d}s,
	\end{align*}
	for any $u\in\mathrm{L}^2(J;\mathbb{U})$.  Next, we define the \emph{admissible class} $\mathscr{A}_{\text{ad}}$  for the system \eqref{linear} as
	\begin{align*}
	\mathscr{A}_{\text{ad}}:=\big\{(x,u) :x\text{ is \text{a unique mild solution} of }\eqref{linear}  \text{ with control }u\in\mathrm{L}^2(J;\mathbb{U})\big\}.
	\end{align*}
	For a given control $u\in\mathrm{L}^2(J;\mathbb{U})$, the system \eqref{linear} has a mild solution, which ensures that the set $\mathscr{A}_{\text{ad}}$ is nonempty. By using the definition of the cost functional, we can formulate the optimal control problem as  :
	\begin{align}\label{eqn:4.7}
	\min_{ (x,u) \in \mathscr{A}_{\text{ad}}}  \mathcal{F}(x,u).
	\end{align}
	An optimal pair for the problem \eqref{eqn:4.7} is denoted by $({x}^0, u^0),$ where $u^0$ denotes an \emph{optimal control}. The existence of optimal pair for the problem \eqref{eqn:4.7} is given by the following theorem:
	\begin{theorem}[Existence of an optimal pair, \cite{MTM}]\label{optimal}
		For a given $\phi(0)\in\mathbb{X}$, there exists at least one pair  $(x^0,u^0)\in\mathscr{A}_{\text{ad}}$  such that the functional $\mathcal{F}(x,u)$ attains its minimum at $(x^0,u^0)$, where $x^0$ is the unique mild solution of the system  \eqref{linear}  with the control $u^0.$
	\end{theorem}
	\begin{rem}
		Since the cost functional defined in \eqref{eqn:4.5} is convex, the constraint \eqref{linear} is linear and the admissible control class $\mathrm{L}^2(J;\mathbb{U})$ is convex, then the optimal control obtained in Theorem \ref{optimal} is unique. 
	\end{rem}
	\noindent The explicit expression of the optimal control ${u}$ is given by the following lemma:
	\begin{lem}[\cite{MTM}]\label{lem3.1}
			For $\lambda>0$, the optimal control $u_\lambda$ satisfying \eqref{linear} and minimizing the cost functional \eqref{eqn:4.5} is given by
			\begin{align*}
			u_\lambda(t)=\mathrm{B}^{*}\mathrm{U}^{*}(T,t)\mathcal{J}\left[\mathrm{R}(\lambda,\Psi_0^{T})p(x(\cdot))\right],\ t\in J,
			\end{align*}
			where
			\begin{align*}
			p(x(\cdot))=x_{T}-\mathrm{U}(T,0)\phi(0).
			\end{align*}
	\end{lem}
	Next, we state a theorem to investigate the approximate controllability of the linear control system \eqref{linear}.  
	\begin{theorem}[Theorem 3.2, \cite{SM}]\label{thm4.2}
			The following statements are equivalent:
			\begin{enumerate}
				\item [(i)]The linear control system \eqref{linear} is approximately controllable on $J$.
				\item [(ii)] If $x^*\in \mathbb{X}^{*}$, we have $\mathrm{B}^{*}\mathrm{U}^*(T,t)x^*=0$, for all $t\in J,$ then $x^*=0.$
				\item [(iii)] The Assumption (H0) holds. 
			\end{enumerate}
	\end{theorem}
	\subsection{Approximate controllability of a semilinear system}
	Here, we establish the semilinear system \eqref{eqn:1.1}  is approximate controllable, whenever the corresponding linear system \eqref{linear} is approximately controllable. To achieve this goal, we first show   the existence of a fixed point of the operator $ \Gamma_\lambda : \mathrm{E}\multimap\mathrm{E} $ (the space $\mathrm{E}$ specified in \eqref{set}) as	
	\begin{align*}
	\Gamma_\lambda(x)&=\bigg\{z\in \mathrm{E}: z(t)=\mathrm{U}(t,0)\phi(0) +\int_{0}^{t}\mathrm{U}(t,s)[\mathrm{B}u_\lambda(s)+f(s)]\mathrm{d}s\\&\qquad+\sum_{0<\tau_{k}<t}\mathrm{U}(t,\tau_{k})I_{k}(\tilde{x}(\tau_{k})), f\in \mathrm{N}_{\mathrm{F}}(x) \bigg\},
	\end{align*} 
	where $\lambda>0$ and  $x_{T}\in \mathbb{X}$.	The control $u_{\lambda}(\cdot)$ is defined as
	\begin{align}
	u_{\lambda}(t)&=\mathrm{B}^{*}\mathrm{U}^{*}(T,t)\mathcal{J}\left[\mathrm{R}(\lambda,\Psi_{0}^{T})g(x(\cdot))\right],\label{eqn:cont}\end{align}
	where 
	\begin{align*}
	g(x(\cdot))&=x_{T}-\mathrm{U}(T,0)\phi(0)-\int^{T}_{0}\mathrm{U}(T,s)f(s)\mathrm{d}s-\sum_{k=1}^{m}\mathrm{U}(T,\tau_{k})I_{k}(\tilde{x}(\tau_{k})),
	\end{align*}
	and $\tilde{x}:(-\infty,T]\rightarrow\mathbb{X}$  such that $\tilde{x}_{0}=\phi$ and $\tilde{x}=x$ on $J$. The existence of a fixed point of the operator $\Gamma_\lambda$ guarantees that the system \eqref{eqn:1.1} has a mild solution with the control \eqref{eqn:cont}. 
	\begin{theorem}\label{thm4.3}
		Let the Assumptions (R1)-(R4) and (H1)-(H5) hold true. Then  for every $ \lambda>0 $ and for fixed $x_{T}\in\mathbb{X},$ the operator $ \Gamma_\lambda : \mathrm{E}\multimap\mathrm{E}$ has a fixed point.
	\end{theorem}
	\begin{proof} We use  Theorem \ref{thm:2.1} to show that the multi-valued operator $\Gamma_\lambda$ has a fixed point.  Similar to the proof of Theorem \ref{thm3.2}, we divide the proof in the following steps.
		\vskip 0.1in 
		\noindent\textbf{Step (1): } \emph{$\Gamma_\lambda$ is convex for each $x\in \mathrm{E}$}. The proof of convexity of the operator $\Gamma_\lambda(x),$ for any $x\in \mathrm{E}$ can be carried out in a similar manner as in Step 1 of the proof of Theorem \ref{thm3.2}.
		\vskip 0.1in 
		\noindent\textbf{Step (2): } \emph{$\Gamma_\lambda$ is bounded}. For this we must show that for any $q>0$, the operator $\Gamma_\lambda$ maps $\mathrm{E}_q$  into a bounded subset of $\mathrm{E}$, where $\mathrm{E}_q:=\{x\in\mathrm{E}:\left\|x\right\|_{\mathrm{PC}_{\mathcal{B}}}\le q\}$. First, we estimate $\| u_{\lambda}(t)\|_{\mathbb{U}},$ by using  (\ref{eqn:cont}), \eqref{eqn:4.4}, and  Assumption \ref{as3.1} as
		\begin{align}\label{48}
		\left\|u_{\lambda}(t)\right\|_{\mathbb{U}}&=\left\|\mathrm{B}^{*}\mathrm{U}^{*}(T,t)\mathcal{J}[\mathrm{R}(\lambda,\Psi_{0}^{T})g(x(\cdot))]\right\|_{\mathbb{U}}\nonumber\\&\leq\frac{1}{\lambda}\left\|\mathrm{B}^*\right\|_{\mathcal{L}(\mathbb{X}^{*},\mathbb{U})}\left\|\mathrm{U}^{*}(T,t)\right\|_{\mathcal{L}(\mathbb{X}^{*})}\left\|\mathcal{J}[\lambda\mathrm{R}(\lambda,\Psi_{0}^{T})g(x(\cdot))]\right\|_{\mathbb{X}^{*}}\nonumber\\&\leq \frac{CM_{B}}{\lambda}\left\|g(x(\cdot))\right\|_{\mathbb{X}}\nonumber\\&\leq\frac{CM_{B}}{\lambda}\bigg(\left\|x_T\right\|_{\mathbb{X}}+\left\|\mathrm{U}(T,0)\right\|_{\mathcal{L}(\mathbb{X})}\left\|\phi(0)\right\|_{\mathbb{X}}+\int_0^T\left\|\mathrm{U}(T,s)\right\|_{\mathcal{L}(\mathbb{X})}\left\|f(s)\right\|_{\mathbb{X}}\mathrm{d}s\nonumber\\&\qquad \qquad +\sum_{k=1}^{m}\left\|\mathrm{U}(T,\tau_{k})\right\|_{\mathcal{L}(\mathbb{X})}\left\|I_{k}(\tilde{x}(\tau_{k}))\right\|_{\mathbb{X}}\bigg)\nonumber\\&\leq\frac{CM_{B}}{\lambda}\left(\left\|x_T\right\|_{\mathbb{X}}+C\left\|\phi(0)\right\|_{\mathbb{X}}+C\int_0^T\gamma(s)\mathrm{d}s+C\sum_{k=1}^{m}d_k\right)\nonumber\\&\leq\frac{CM_{B}}{\lambda}\left(\left\|x_T\right\|_{\mathbb{X}}+C\left\|\phi(0)\right\|_{\mathbb{X}}+C\left\|\gamma\right\|_{\mathrm{L}^1(J;[0,+\infty))}+C\sum_{k=1}^{m}d_k\right)\nonumber\\&\leq\frac{CMM_{B}}{\lambda},
		\end{align}
		for all $t\in J$, where $M=\left\|x_T\right\|_{\mathbb{X}}+C\left\|\phi(0)\right\|_{\mathbb{X}}+C\left\|\gamma\right\|_{\mathrm{L}^1(J;[0,+\infty))}+C\sum_{k=1}^{m}d_k$. Let us take $x\in \mathrm{E}_q$ and $z\in\Gamma_\lambda(x)$, then there exists $f\in\mathrm{N}_\mathrm{F}(x)$ such that 
		\begin{align}\label{eqn:4.10}
		z(t)=\mathrm{U}(t,0)\phi(0)+\int_{0}^{t}\mathrm{U}(t,s)\left[\mathrm{B}u_\lambda(s)+f(s)\right]\mathrm{d}s+\sum_{0<\tau_{k}<t}\mathrm{U}(t,\tau_{k})I_{k}(\tilde{x}(\tau_{k})).
		\end{align}
		We now compute 
		\begin{align*}
		&\left\|z(t)\right\|_{\mathbb{X}}\\&\le\left\|\mathrm{U}(t,0)\phi(0)\right\|_{\mathbb{X}}+\int_{0}^{t}\left\|\mathrm{U}(t,s)\left[\mathrm{B}u_\lambda(s)+f(s)\right]\mathrm{d}s\right\|_{\mathbb{X}}+\sum_{0<\tau_{k}<t}\left\|\mathrm{U}(t,\tau_{k})I_{k}(\tilde{x}(\tau_{k}))\right\|_{\mathbb{X}}\nonumber\\&\le \left\|\mathrm{U}(t,0)\right\|_{\mathcal{L}(\mathbb{X})}\left\|\phi(0)\right\|_{\mathbb{X}}+\int_{0}^{t}\left\|\mathrm{U}(t,s)\right\|_{\mathcal{L}(\mathbb{X})}\left[\left\|\mathrm{B}\right\|_{\mathcal{L}(\mathbb{U},\mathbb{X})}\left\|u_\lambda(s)\right\|_{\mathbb{U}}+\left\|f(s)\right\|_{\mathbb{X}}\right]\mathrm{d}s\nonumber\\&\quad+\sum_{0<\tau_{k}<t}\left\|\mathrm{U}(t,\tau_{k})\right\|_{\mathcal{L}(\mathbb{X})}\left\|I_{k}(\tilde{x}(\tau_{k}))\right\|_{\mathbb{X}}\nonumber\\&\le C\left\|\phi(0)\right\|_{\mathbb{X}}+\frac{C^2M_B^2MT}{\lambda}+C\left\|\gamma\right\|_{\mathrm{L}^1(J;[0, +\infty))}+C\sum_{k=1}^{m}d_k,
		\end{align*}
		for all $t\in J$. Here, we  used the condition (1) of Lemma \ref{lem2.1}, Assumption \ref{as3.1} \textit{(H1)}, \textit{(H4)} and \textit{(H5)}. Thus, the operator $\Gamma_\lambda$ is bounded. 
		\vskip 0.1in 
		\noindent\textbf{Step (3):} \emph{$\Gamma_\lambda$ is completely continuous}.  Let us take $x\in \mathrm{E}_q$ and $z\in\Gamma_\lambda(x)$, then there exists $f\in\mathrm{N}_\mathrm{F}(x)$ such that \eqref{eqn:4.10} holds. Using the estimate \eqref{48} of $\left\|u_\lambda(t)\right\|_{\mathbb{U}},$ for all $t\in J$, and then following the Steps (3) and (4)  in the proof of Theorem \ref{thm3.2}, we obtain that the multi-valued map $\Gamma_\lambda$ is completely continuous.
		\vskip 0.1in 
		\noindent\textbf{Step (4):} \emph{$\Gamma_\lambda$ is upper semicontinuous}.  From Step (3), it is clear that the operator $\Gamma_\lambda$  is completely continuous.  To prove the operator $\Gamma_\lambda$ is u.s.c, it suffices to prove that $\Gamma_\lambda$ has a closed graph (see Remark \ref{rem2.1}). For this, we consider a sequences $x^n\to x^*$ in $\mathrm{E}$ and $z^n\to z^*$ in $\mathrm{E}$ with $z^n\in\Gamma_\lambda(x^n)$, then there exist $f_n\in \mathrm{N}_\mathrm{F}(x^n) $ such that 
		\begin{align}\label{eqn:4.12}
		z^n(t)= \mathrm{U}(t,0)\phi(0)+\int_{0}^{t}\mathrm{U}(t,s)\left[\mathrm{B}u_{\lambda,n}(s)+f_n(s)\right]+\sum_{0<\tau_{k}<t}\mathrm{U}(t,\tau_{k})I_{k}(\tilde{x}(\tau_{k})), 
		\end{align}
		where 
		\begin{align*}
		u_{\lambda,n}(t)&=\mathrm{B}^{*}\mathrm{U}^{*}(T,t)\mathcal{J}\left[\mathrm{R}(\lambda,\Psi_{0}^{T})g_n(x(\cdot))\right],\end{align*}
		with
		\begin{align*}
		g_n(x(\cdot))&=x_{T}-\mathrm{U}(T,0)\phi(0)-\int^{T}_{0}\mathrm{U}(T,s)f_n(s)\mathrm{d}s-\sum_{k=1}^{m}\mathrm{U}(T,\tau_{k})I_{k}(\tilde{x}(\tau_{k})).
		\end{align*}
		It is straightforward by the Assumption \ref{as3.1}  \textit{(H4)} that the sequence $\{f_n\}_{n\ge 1}$ is integrably bounded. Since the multi-valued map $\mathrm{F}$ is weakly compact valued and $f_n(t)\in \mathrm{F}(t, \tilde{x_t^n})$, then by invoking the Dunford-Pettis theorem, we see that $\{f_n\}_{n\ge 1}$ is relatively weakly compact in $\mathrm{L}^1(J;\mathbb{X})$. Hence, we can find a subsequence of $\{f_n\}_{n\ge 1}$ relabeled as $\{f_n\}_{n\ge 1}$ such that 
		\begin{align}\label{eqn:4.13}
		f_n\xrightharpoonup{w}  f^{*} \ \text{in}\ \mathrm{L}^1(J;\mathbb{X}), \ \text{ as }\ n\to\infty. 
		\end{align}
		Let us define 
		\begin{align*}
		u^*_{\lambda}(t)&=\mathrm{B}^{*}\mathrm{U}^{*}(T,t)\mathcal{J}\left[\mathrm{R}(\lambda,\Psi_{0}^{T})g^*(x(\cdot))\right],\end{align*}
		where
		\begin{align*}
		g^*(x(\cdot))&=x_{T}-\mathrm{U}(T,0)\phi(0)-\int^{T}_{0}\mathrm{U}(T,s)f^{*}(s)\mathrm{d}s-\sum_{k=1}^{m}\mathrm{U}(T,\tau_{k})I_{k}(\tilde{x}(\tau_{k})).
		\end{align*}
		By using \eqref{eqn:4.4}, \eqref{eqn:4.13} and Corollary \ref{cor:3.1}, we deduce that
		\begin{align*}
		\left\|\mathrm{R}(\lambda,\Psi_{0}^{T})g_n(x(\cdot))-\mathrm{R}(\lambda,\Psi_{0}^{T})g^*(x(\cdot))\right\|_{\mathbb{X}}&=\frac{1}{\lambda}\left\|\lambda\mathrm{R}(\lambda,\Psi_{0}^{T})\left(g_n(x(\cdot))-g^*(x(\cdot))\right)\right\|_{\mathbb{X}}\nonumber\\&\leq\frac{1}{\lambda}\left\|g_n(x(\cdot))-g^*(x(\cdot))\right\|_{\mathbb{X}}\nonumber\\&\leq\left\|\int^{T}_{0}\mathrm{U}(T,s)\left[f_n(s)-f^{*}(s)\right]\mathrm{d}s\right\|_{\mathbb{X}}\nonumber\\&\to0,\ \text{as}\ n\to\infty,
		\end{align*}
		and hence $\mathrm{R}(\lambda,\Psi_{0}^{T})g_n(x(\cdot))\to \mathrm{R}(\lambda,\Psi_{0}^{T})g^*(x(\cdot))$ in $\mathbb{X}$ as $n\to\infty$. Since the mapping $\mathcal{J}:\mathbb{X}\to\mathbb{X}^*$  is  demicontinuous, it is immediate that 
		\begin{align}\label{eqn:4.14}
		\mathcal{J}\left[\mathrm{R}(\lambda,\Lambda_{T})g_n(x(\cdot))\right]\xrightharpoonup{w}\mathcal{J}\left[\mathrm{R}(\lambda,\Lambda_{T})g^*(x(\cdot))\right] \ \text{ as } \ n\to\infty  \ \text{ in }\ \mathbb{X}^*.
		\end{align}
		Since $\mathrm{U}(t,s)$ is compact for $t>s$ and $\mathrm{B}$ is a bounded linear operator from $\mathbb{U}$ to $\mathbb{X}$, we obtain that the operator $\mathrm{U}(t,s)\mathrm{B}\mathrm{B}^*\mathrm{U}^*(t,s),$ $t>s$, is a compact operator from  $\mathbb{X}$ into itself. By using the Lemma 4.1, \cite{MTM}, one can show that the operator $$\varphi\mapsto\int_0^t\mathrm{U}(t,s)\mathrm{B}\mathrm{B}^*\mathrm{U}(t,s)\varphi(s)\mathrm{d}s$$ is a compact operator from $\mathrm{L}^2(J;\mathbb{X})\to\mathrm{C}(J;\mathbb{X})$. Combining these facts with the weak convergence \eqref{eqn:4.14}, we deduce that
		\begin{align}\label{eqn:4.15}
		&\left\|\int_0^t\mathrm{U}(t,s)\mathrm{B}\mathrm{B}^{*}\mathrm{U}^{*}(T,t)\left\{\mathcal{J}\left[\mathrm{R}(\lambda,\Lambda_{T})g_n(x(\cdot))\right]-\mathcal{J}\left[\mathrm{R}(\lambda,\Lambda_{T})g^*(x(\cdot))\right]\right\}\mathrm{d}s\right\|_{\mathbb{X}}\nonumber\\&\to 0 \ \text{ as } \ n\to\infty. 
		\end{align}
		By passing $n\to\infty$ in \eqref{eqn:4.12}, we get
		\begin{align}\label{eqn:4.16}
		z^n(t)\to\mathrm{U}(t,0)\phi(0)+\int_{0}^{t}\mathrm{U}(t,s)\left[\mathrm{B}u^*_\lambda(s)+f^{*}(s)\right]+\sum_{0<\tau_{k}<t}\mathrm{U}(t,\tau_{k})I_{k}(\tilde{x}(\tau_{k}))=z^0(t),\end{align}
		where we  used the convergences \eqref{eqn:4.13}, \eqref{eqn:4.15} and Corollary \ref{cor:3.1}. Since we know that $x^n\to x^*$ in $\mathrm{E}$ and $f_n\in\mathrm{N}_\mathrm{F}(x^n)$,  using \eqref{eqn:4.13} and then applying Lemma \ref{thm:3.1}, we obtain  $f^{*}\in\mathrm{N}_\mathrm{F}(x^*)$, which implies that  $z^0\in\Gamma_\lambda(x^*)$. Moreover, the convergence \eqref{eqn:4.16} ensures that  $z^0(t)= z^*(t),$ for $ t\in J$. Therefore, we have $z^*\in\Gamma_\lambda(x^*)$ and hence, the graph of the operator $\Gamma_\lambda$ is closed.
		\vskip 0.1in 
		\noindent\textbf{Step (6):} \emph{The set $\mathcal{D}_{\lambda}:=\{x\in\mathrm{E}: x\in\kappa\Gamma_\lambda(x), \ 0<\kappa<1\}$ is bounded}. Using the estimate \eqref{48} of $\left\|u_\lambda(t)\right\|_{\mathbb{U}},$ for all $t\in J$ and then following the Step 6 of the proof of Theorem \ref{thm3.2}, we can easily verify that the set $\mathcal{D}_{\lambda}$ is bounded.
		
		
		Then, Theorem \ref{thm:2.1} yields that the operator $\Gamma_\lambda$ has a fixed point, which implies that the system \eqref{eqn:1.1} has a mild solution.
	\end{proof}
	\begin{theorem}\label{thm4.4}
		Let the Assumptions (R1)-(R4), (H0)-(H5) hold true.  Then, the system (\ref{eqn:1.1}) is approximately controllable on $J$.
	\end{theorem}
	\begin{proof}
		Let $x^{\lambda}(\cdot)$ be a fixed point of the operator $\Gamma_\lambda,$ then $x^{\lambda}(\cdot)$ is a mild solution of the equation \eqref{eqn:1.1} with the control 
		\begin{align}
		u_{\lambda}(t)&=\mathrm{B}^{*}\mathrm{U}^{*}(T,t)\mathcal{J}\left[\mathrm{R}(\lambda,\Psi_{0}^{T})g(x^\lambda(\cdot))\right],\label{cont:24}
		\end{align}
		\begin{align*}
		g(x^\lambda(\cdot))&=x_{T}-\mathrm{U}(T,0)\phi(0)-\int^{T}_{0}\mathrm{U}(T,s)f^\lambda(s)\mathrm{d}s-\sum_{k=1}^{m}\mathrm{U}(T,\tau_{k})I_{k}(\tilde{x^\lambda}(\tau_{k})),
		\end{align*}
		where $f^\lambda\in\mathrm{N}_\mathrm{F}(x^\lambda)$.
		That is, $x^{\lambda}(\cdot)$ satisfies 
		\begin{align*}
		x^{\lambda}(t)= \mathrm{U}(t,0)\phi(0) +\int_{0}^{t}\mathrm{U}(t,s)[\mathrm{B}u_{\lambda}(s)+f^\lambda(s)]\mathrm{d}s+\sum_{0<\tau_{k}<t}\mathrm{U}(t,\tau_{k})I_{k}(\tilde{x^{\lambda}}(\tau_{k})),
		\end{align*}
		with the control $u_{\lambda}(\cdot)$ given in \eqref{cont:24}.
		It is easy to verify that 
		\begin{align}\label{eqn:4.18}
		x^{\lambda}(T)= x_{T}-\lambda \mathrm{R}(\lambda, \Psi_{0}^{T})g(x^{\lambda}(\cdot)).
		\end{align}
		Using the Assumption \ref{as3.1} \textit{(H4)}  and Dunford-Pettis theorem, we know that the sequence $\{f^{\lambda}\}$ is weakly compact in $\mathrm{L}^1(J;\mathbb{X})$.  Therefore, there exists a subsequence, still denoted as $\{f^{\lambda}\}$, such that  
		\begin{align*}
		f^\lambda\xrightharpoonup{w}  f^{*} \ \text{ in }\ \mathrm{L}^1(J;\mathbb{X}) \ \mbox{ as }\ \lambda\to0^+.
		\end{align*}
		Moreover, by the Assumption \ref{as3.1} \textit{(H5)}, we infer that the sequence $\{I_{k}(\tilde{x^\lambda})(\tau_{k}) : \lambda > 0\} $ is bounded in $ \mathbb{X}$, for $ k=1,\ldots,m.$ By invoking the Banach-Alaoglu theorem, we can find weakly convergent subsequences relabeled as $\{I_{k}(\tilde{x^\lambda})(\tau_{k}) : \lambda > 0\},$ with pointwise weak limit $\eta_{k}\in\mathbb{X},$ for each $k=1,\ldots,m.$  We now evaluate
		\begin{align}\label{eqn:1.17}
		&\left\|g(x^{\lambda}(\cdot))-\omega\right\|_{\mathbb{X}}\nonumber\\&\leq\left\|\int_{0}^{T}\mathrm{U}(T,s)[f^\lambda(s)-f(s)]\mathrm{d}s\right\|_{\mathbb{X}}+\left\|\sum_{k=1}^{m}\mathrm{U}(T,\tau_k)\left[I_{k}(\tilde{x^{\lambda}}(\tau_{k}))-\eta_k\right]\right\|_{\mathbb{X}}\nonumber\\&\leq\left\|\int_{0}^{T}\mathrm{U}(T,s)[f^\lambda(s)-f(s)]\mathrm{d}s\right\|_{\mathbb{X}}+\sum_{k=1}^{m}\left\|\mathrm{U}(T, \tau_k)\left[I_{k}(\tilde{x^{\lambda}}(\tau_{k}))-\eta_k\right]\right\|_{\mathbb{X}}\nonumber\\&\to0 \text { as } \lambda\to0^{+},
		\end{align}
		where
		\begin{align*}
		\omega= x_{T}-\mathrm{U}(T,0)\phi(0)-\int_{0}^{T}\mathrm{U}(T, s)f(s)ds-\sum_{k=1}^{m}\mathrm{U}(T,\tau_{k})\eta_k.
		\end{align*}
		The first term in the right hand side of \eqref{eqn:1.17} goes to zero using  Corollary \ref{cor:3.1} and  the final term tends to zero using the compactness of the evolution system $\mathrm{U}(t,s)$ for $t-s>0$.
		
		Finally, by using the equality \eqref{eqn:4.18}, we estimate $\left\|x^{\lambda}(T)-x_{T}\right\|_{\mathbb{X}}$ as 
		\begin{align*}
		\nonumber\left\|x^{\lambda}(T)-x_{T}\right\|_{\mathbb{X}}&=\left\|\lambda \mathrm{R}(\lambda, \Psi_{0}^{T})g(x^{\lambda}(\cdot))\right\|_{\mathbb{X}}\\
		&\leq \left\|\lambda \mathrm{R}(\lambda, \Psi_{0}^{T})\omega\right\|_{\mathbb{X}}+\left\|\lambda  \mathrm{R}(\lambda, \Psi_{0}^{T})(g(x^{\lambda}(\cdot))-\omega)\right\|_{\mathbb{X}}.
		\end{align*}
		Using the estimate \eqref{eqn:1.17} and the Assumption \ref{ass4.1} \textit{(H0)},  we easily obtain 
		\begin{align*}
		\left\|x^{\lambda}(T)-x_{T}\right\|_{\mathbb{X}} \rightarrow 0  \mbox{ as }  \lambda \rightarrow 0^{+},
		\end{align*}
		which ensures that the system (\ref{eqn:1.1}) is  approximately controllable.
	\end{proof}
	\section{Application}\label{sec5}\setcounter{equation}{0}
In this section, we provide an example to validate the results obtained in the previous sections.
	\begin{Ex}\label{ex1} Let us consider a control system governed by the functional impulsive differential inclusions of the form
		\begin{equation}\label{57}
		\left\{
		\begin{aligned}
		\frac{\partial y(t,\xi)}{\partial t}&\in a(t)\frac{\partial^2y(t,\xi)}{\partial \xi^2}+\mathrm{B}u(t,\xi)+F\left(t,y(t-r,\xi)\right), \\&\qquad \qquad r>0,\ t\in J=[0,T], \ t\neq \tau_{k},\ k=1,\dots,m, \ \xi\in[0,\pi], \\
		y(t,0)&=0=y(t,\pi), \  t\in J=[0, T], \\
		y(\theta,\xi)&=\phi(\theta,\xi), \ \xi\in[0,\pi], \ \theta\in(-\infty, 0],\\
		\Delta y(\tau_{k},\xi)&= \int_{0}^{\pi}\rho_{k}(\xi, \eta)\cos^{2}(y(\tau_{k}, \eta))d\eta, \ k=1,\ldots,m, \ \xi\in[0,\pi],
		\end{aligned}
		\right.
		\end{equation}
		where $a:J\to\mathbb{R}^+,$ is  H\"older continuous function of order $0<\kappa\leq 1$, that is, there exists a positive constant $C_a$ such that 
		$$\vert a(t)-a(s)\vert\leq C_a\vert t-s \vert^{\kappa}, \ \text{ for all } \ t,s\in J,$$ and $ \rho_{k} \in C([0, \pi]\times [0, \pi]; \mathbb{R})$.
	\end{Ex}
	\vskip 0.1 cm
	\noindent\textbf{Step 1:} \emph{Evolution family and phase space.}
	Let $\mathbb{X}_p=\mathrm{L}^{p}([0,\pi];\mathbb{R})$, for $p\in[2,\infty)$ and the family of operators $\mathrm{A}_p(t)$ defined as $\mathrm{A}_p(t)f(\xi)=a(t)f''(\xi),$ with the domain $\mathrm{D}(\mathrm{A}_p(t))=\mathrm{D}(\mathrm{A}_p)=\mathrm{W}^{2,p}([0,\pi];\mathbb{R})\cap\mathrm{W}_0^{1,p}([0,\pi];\mathbb{R})$. We define the operator $\mathrm{A}_p$ as
	\begin{align*}
	\mathrm{A}_pf(\xi)=f''(\xi), \ \ \xi\in [0,\pi],
	\end{align*}
	with the domain $\mathrm{D}(\mathrm{A}_p)$. Moreover, for  $t\in J$ and $f\in\mathrm{D}(\mathrm{A}_p) $, the operator $\mathrm{A}_p(t)$ can be written as
	\begin{align*}
	\mathrm{A}_p(t)f= \sum_{n=1}^{\infty}(-n^{2}a(t))\langle f, w_{n} \rangle  w_{n},\ f\in \mathrm{D}(\mathrm{A}_p), \text{ with } \langle f,w_n\rangle =\int_0^{\pi}f(\xi)w_n(\xi)\mathrm{d}\xi,
	\end{align*}
	where, $-n^2$($n\in\mathbb{N}$) and $w_n(\xi)=\sqrt{\frac{2}{\pi}}\sin(n\xi)$, are the eigenvalues and the  corresponding normalized eigenfunctions of the operator $\mathrm{A}_p$ respectively. The operator  $\mathrm{A}_p(t)$ fulfills all the conditions (\textit{R1})-(\textit{R4}) of the Assumption \ref{ass2.1} (see, application section of \cite{MTM}).
	Then by applying Lemma \ref{lem2.1}, we obtain the existence of a unique evolution system $\left\{\mathrm{U}_p(t,s):0\leq s\leq t\leq T\right\}$. From Lemma $ \ref{lem2.2}$, one can observe that the evolution system $\left\{\mathrm{U}_p(t,s):0\leq s\leq t\leq T\right\}$ is compact for $t-s>0.$ The evolution system  $\mathrm{U}_p(t,s)$ can be explicitly written as
	\begin{align*}
	\mathrm{U}_p(t,s)f= \sum_{n=1}^{\infty}e^{-n^{2}\int_{s}^{t}a(\tau)\mathrm{d}\tau}\langle f, w_{n} \rangle w_{n}, \text{ for each }  f\in\mathbb{X}_p.
	\end{align*}
	Moreover, we have 
	\begin{align*}
	\mathrm{U}^*_p(t,s)f^*= \sum_{n=1}^{\infty}e^{-n^{2}\int_{s}^{t}a(\tau)\mathrm{d}\tau}\langle f^*, w_{n} \rangle w_{n}, \text{ for each }  f^*\in\mathbb{X}^*_p.
	\end{align*}
	Now, we define the phase space as \begin{align*}
	\mathcal{B}_{g,p}=\left\{\phi:(-\infty, 0]\to\mathbb{X}_p:\text{for}\  r'>0,\ \phi\vert_{[-r', 0]}\in\mathcal{B}_p \text{ and } \int_{-\infty}^{0}\!\!\!\!\!\!g(\theta)\left\|\phi\right\|_{[\theta, 0]}\mathrm{d}\theta<+\infty \right\},
	\end{align*}
	endowed with the norm $$\left\|\phi\right\|_{\mathcal{B}_{g,p}}:=\int_{-\infty}^{0}g(\theta)\left\|\phi\right\|_{[\theta, 0]}\mathrm{d}\theta, \ \text{ for all }\ \phi\in\mathcal{B}_{g,p},$$ where
	$\mathcal{B}_p=\left\{\phi:[-r, 0]\to\mathbb{X}_p, r>0:\ \phi\ \text{is bounded and measurable}\right\}$, provided with the norm $\left\|\phi\right\|_{[-r, 0]}:=\int_{-r}^{0}\left\|\phi(\theta)\right\|_{\mathbb{X}_p}\mathrm{d}\theta$ and $g(\theta)=e^{\nu\theta},\ \theta<0,\nu>0$.
	\vskip 0.1 in
	\noindent\textbf{Step 2:} \emph{Abstract formulation and approximate controllability.}
	Let us define $$x(t)(\xi):=y(t,\xi),\ \text{ for }\ t\in J\ \text{ and }\ \xi\in[0,\pi],$$ and the bounded linear operator $\mathrm{B}:\mathbb{U}\to\mathbb{X}_p$  as
	\begin{align*}
	\mathrm{B}u:=2u_2w_1+\sum_{n=2}^{\infty}u_nw_n, \ \text{ for all } \ u\in\mathbb{U},
	\end{align*}
	where $u_n=(u,w_n)$ and the control space 
	\begin{align*}
	\mathbb{U}:=\left\{u: u=\sum_{n=2}^{\infty}u_nw_n \ \text{with}\ \sum_{n=2}^{\infty}u_n^2< +\infty\right\},
	\end{align*}
	equipped with the norm $\left\|u\right\|_{\mathbb{U}}=\left(\sum_{n=2}^{\infty}u_n^2\right)^2$.
	It is easy to verify that
	\begin{align*}
	\mathrm{B}^*x^*=(2x_1+x_2)w_2+\sum_{n=3}^{\infty}x_nw_n, \ \text{ for all }\ x^*=\sum_{n=1}^{\infty}x_nw_n\in\mathbb{X}^*_p.
	\end{align*}
	We assume the following assumptions on the multi-valued map $\mathrm{F}:J\times\mathbb{R}\multimap\mathbb{R}$. 
	\begin{enumerate}
		\item The function $\mathrm{F}$ is weakly compact and convex valued.
		\item For any $r>0$, the function $\mathrm{F}\left(t,y(t-r)(\xi)\right)$ has a measurable selection for $t\in J$ and each fixed $y\in\mathbb{X}_p,\ p\in[2,\infty)$, where $\xi\in[0,\pi]$.
		\item For any $r>0$, the function $\mathrm{F}\left(t,y(t-r)(\xi)\right)$ is u.h.c in $y\in\mathbb{X}_p,\ p\in[2,\infty)$, for a.e. $t\in J$, where $\xi\in[0,\pi]$.
		\item  There exist a function $\mu\in\mathrm{L}^1(J;[0,+\infty))$ such that $$ \int_{0}^{\pi}\vert \mathrm{F}(t, y(t-r)(\xi))\vert^{p}\mathrm{d}\xi\leq \mu(t), \ \mbox{for a.e.}\ t\in J,\ r>0,\ \xi\in[0,\pi],\ p\in[2,\infty).$$
	\end{enumerate}
	We define the multi-valued map $\mathrm{F}:J\times\mathcal{B}_{g}\multimap\mathbb{X}_p$ such that
	$$\mathrm{F}(t,x_t)(\xi):=\mathrm{F}(t,y(t-r,\xi) ), \mbox{for}\ r>0,\ \xi\in[\pi,0],\ y\in\mathbb{X}_p.$$
	It is clear from the above facts that the multi-valued map $\mathrm{F}$ satisfies the conditions (\textit{H2})-(\textit{H4}) of the Assumption \ref{as3.1}. Next, we define the impulses $I_{k}:\mathbb{X}_p\to\mathbb{X}_p,$ for $k=1,\ldots,m$  as 
	\begin{align*}
	I_{k}(x)(\xi)&=\int_{0}^{\pi}\rho_{k}(\xi, \eta)\cos^{2}(x(\eta))d\eta,\ \mbox{for} \ k=1,\ldots,m,\ \xi\in[0, \pi].
	\end{align*}
	Thus, the impulses  $I_{k}$, for $k=1,\ldots,m$ satisfy the condition (\textit{H5}) of the Assumption \ref{as3.1} (see \cite{SA}). 
	
	The system \eqref{57} can be expressed as an abstract form given in (\ref{eqn:1.1}) by using the above substitutions and it satisfies the Assumption \ref{as3.1}. Moreover, it remains to show  that the linear system corresponding to the system \eqref{eqn:1.1} is approximately controllable. For this, we need to prove that for any $x^*\in\mathbb{X}_p^*$ such that  $\mathrm{B}^*\mathrm{U}_p^*(T,t)x^*=0$ implies $x^*=0$. To complete the proof, we evaluate 
	\begin{align*}
	\mathrm{B}^*\mathrm{U}_p^*(T,t)x^*=\left(2x_1e^{-\int_{t}^{T}a(\tau)\mathrm{d}\tau}+x_2e^{-4\int_{t}^{T}a(\tau)\mathrm{d}\tau}\right)w_2+\sum_{n=3}^{\infty}e^{-n^2\int_{t}^{T}a(\tau)\mathrm{d}\tau}x_nw_n,
	\end{align*}
	for all $x^*=\sum_{n=1}^{\infty}x_nw_n\in\mathbb{X}^*_p$.  It follows that $\left\|\mathrm{B}^*\mathrm{U}_p^*(T,t)x^*\right\|_{\mathbb{U}}=0,$ for certain $t\in J$, implies $x^*=0$.  Hence, by Theorem \ref{thm4.2}, we obtain that the linear system corresponding to \eqref{eqn:1.1} is approximately controllable and the Assumption \ref{ass4.1} (\textit{H0}) holds. Finally, by applying  Theorem \ref{thm4.4}, we can conclude that the semilinear system \eqref{eqn:1.1} (equivalent to the system \eqref{57}) is approximately controllable.
	\section{Conclusions} In this article, we first investigated the existence of a mild solution of the differential inclusion \eqref{eqn:1.1} and then proved the approximate controllability of the problem \eqref{eqn:1.1} via resolvent operator condition and a generalization of the Leray-Schauder fixed point theorem for multi-valued maps. The conditions considered (Assumptions \ref{ass2.1} and \ref{as3.1}) for obtaining the approximate controllability of the problem \eqref{eqn:1.1}  are sufficient only but not necessary.  Moreover, the resolvent operator in Banach spaces is properly defined  and modified the axioms of the phase space to deal with  impulsive functional differential equations and inclusions with infinite delay. Furthermore, we proved the compactness of the operator $h(\cdot)\mapsto \int_{0}^{\cdot}\mathrm{U}(\cdot,s)h(s)\mathrm{d} s : \mathrm{L}^{1}([0,T];\mathbb{Y}) \rightarrow \mathrm{C}([0,T];\mathbb{Y}),$ for integrably bounded sequences in $\mathrm{L}^{1}([0,T];\mathbb{Y})$ ($\mathbb{Y}$ is a general Banach space).
	
	\medskip\noindent
	{\bf Acknowledgments:} The first author would like to thank Council of Scientific and Industrial Research, New Delhi, Government of India (File No. 09/143(0931)/2013 EMR-I), for financial support to carry out his research work and Department of Mathematics, Indian Institute of Technology Roorkee (IIT Roorkee), for providing stimulating scientific environment and resources. M. T. Mohan would  like to thank the Department of Science and Technology (DST), Govt of India for Innovation in Science Pursuit for Inspired Research (INSPIRE) Faculty Award (IFA17-MA110). 
	
\end{document}